\documentclass[reqno,12pt]{amsart}
\usepackage[letterpaper,top=2cm,bottom=2cm,left=3cm,right=3cm,marginparwidth=1.75cm]{geometry}
\usepackage[dvipsnames]{xcolor}
\usepackage{tabularx} 
\usepackage{amsmath}
\usepackage{graphicx} 
\usepackage{amsmath}
\usepackage{mathtools} 
\usepackage{amsfonts}
\usepackage{amssymb}
\usepackage{amsthm}
\usepackage[final]{hyperref}
\usepackage{marvosym}
\usepackage{mathrsfs}
\usepackage{enumitem}
\usepackage{dsfont}
\usepackage{tikz-cd} 
\usepackage{comment}
\usepackage[utf8]{inputenc}
\usepackage{changepage}

\usepackage[format=plain, font=footnotesize]{caption}
\usepackage[sort]{cite}
\usepackage{xcolor}

\hypersetup{
	colorlinks=true,       % false: boxed links; true: colored links
	linkcolor=teal,        % color of internal links
	citecolor=orange,        % color of links to bibliography
	filecolor=magenta,     % color of file links
	urlcolor=blue         
}
\usepackage{blindtext}

\newcommand{\calV}{\mathcal{V}}

\newcommand{\calH}{\mathcal{H}}

\newcommand{\calM}{\mathcal{M}}

\newcommand{\rank}{\mathrm{rank}\:}

\newcommand{\Gr}{\mathrm{Gr}}
\newcommand{\PP}{\mathbb{P}}
\newcommand{\RR}{\mathbb{R}}

\newcommand{\CC}{\mathbb{C}}
\newcommand{\NN}{\mathbb{N}}
\newcommand{\ZZ}{\mathbb{Z}}

\newcommand{\Ca}{\textit{\textbf{C}}}
\newcommand{\Pck}{\mathcal{P}_{\textit{\textbf c},k}}
\newcommand{\Mck}{\mathcal{M}_{\textit{\textbf c},k}}
\newcommand{\MCk}{\mathcal{M}_{\Ca,k}}

\newcommand{\Vck}{\mathcal{V}_{\textit{\textbf c},k}}
\newcommand{\Hck}{\mathcal{H}_{\textit{\textbf c},k}}

\newcommand{\Gck}
{\overline{\Gamma}_{\textit{\textbf c},k}}
\newcommand{\Phick}{\Phi_{\textit{\textbf c},k}}
\newcommand{\lck}{\ell_{\textit{\textbf c},k}}
\newcommand{\uck}{\upsilon_{\textit{\textbf c},k}}
\newcommand{\ca}{\textit{\textbf c}}

\usepackage[capitalise, noabbrev, nameinlink]{cleveref}

\theoremstyle{plain}

\newtheorem{theorem}{Theorem}[section] 
\newtheorem{proposition}[theorem]{Proposition}
\newtheorem{lemma}[theorem]{Lemma}

\newtheorem{corollary}[theorem]{Corollary}

\theoremstyle{definition}
\newtheorem{definition}[theorem]{Definition} 
\theoremstyle{plain}
\newtheorem{example}[theorem]{Example}
\newtheorem{remark}[theorem]{Remark}

\newtheorem{introtheorem}{Theorem} 
\newtheorem{introproposition}[introtheorem]{Proposition}

\usepackage[ruled,vlined]{algorithm2e}
\usepackage{protosem}
\usepackage{ marvosym }
\usepackage{ulem}
\usepackage{enumitem}

\begin{document}

\title[Projections of Higher Dimensional Subspaces]{Projections of Higher Dimensional Subspaces and Generalized Multiview Varieties}

\thispagestyle{empty}
\date{}

\author{Felix Rydell}

\begin{abstract} We present a generalization of multiview varieties as closures of images obtained by projecting subspaces of a given dimension onto several views, from the photographic and geometric points of view. Motivated by applications in Computer Vision for triangulation of world features, we investigate when the associated projection map is generically injective; an essential requirement for successful triangulation. We give a complete characterization of this property by determining two formulae for the dimensions of these varieties. Similarly, we describe for which center arrangements calibration of camera parameters is possible. We explore when the multiview variety is naturally isomorphic to its associated blowup. In the case of generic centers, we give a precise formula for when this occurs. %We end by considering camera arrangements that can be used for calibration in contrast to triangulation, and state a result relating these two concepts.   
%In the process  determine the dimension of multiview varieties. For this we consider a suitable notion of genericity of camera arrangements, called pseudo-disjointedness. Under this assumption, we prove a formula for the dimension of multiview varieties. 
\end{abstract}

\maketitle
%\keywords{3D reconstruction, triangulation, Algebraic Vision, multiview varieties, line correspondences.}

\setcounter{tocdepth}{2}
%\tableofcontents

\pagenumbering{arabic}

%%%%%%%%%%%%%%%%%%%%%%%%%%%%%%%%%%%%%%%%%%%%%%%%%%%%%%%%%%%%%%%%%%%%%%%%%%%%%%%%%%%%%%%%%%%%%%%%%%%%%%%%%%%%%%%%%%%%%%%%%%%%%%%%%%%%%%%%%%%%%%%%%%%%

\section*{Introduction} 
Computer Vision is the study of how computers can understand images and videos as well as or better than humans do. A fundamental problem within Computer Vision is Structure-from-Motion, which aims to create 3D computer models from 2D images captured by cameras with unknown parameters. This technique finds applications in various domains, including autonomous vehicles \cite{song2015high}, city modelling \cite{agarwal2010reconstructing}, geosciences \cite{carrivick2016structure} and object and human modeling for movies and video games. Commonly, given a set of $n$ 2D images, this process starts by identifying matching features, called \textit{correspondences}, across the images that are recognizable as originating from the same 3D feature. These correspondences are then utilized for camera \textit{calibration}, estimating the camera parameters, and subsequently, the original 3D features are obtained through \textit{triangulation}, which finds the 3D feature that minimizes the reprojection error. See the textbooks \cite{Hartley2004,szeliski2022computer} for more details.

The interplay between Computer Vision and Algebraic Geometry, known as \textit{Algebraic Vision}, arises naturally as a consequence of the fact that cameras are represented as rational maps between complex projective spaces and image correspondences typically involve algebraic objects such as points, lines, and conics. Central to the discussion are \textit{varieties}, which in this paper refers to Zariski closed sets. The study of multiview varieties, which form a key component of Algebraic Vision, dates back to 1997 with the pioneering work by Heyden and Åström \cite{heyden1997algebraic}. Given a camera arrangement $\textit{\textbf{C}}$ comprising of full rank $3\times 4$ matrices $C_1, \ldots, C_n$, that work introduced the \textit{joint image} map
\begin{align}\label{eq: Map Point}\begin{aligned}
    \Phi_{\Ca}: \PP^3&\dashrightarrow (\PP^2)^n,\\
    X& \mapsto (C_1X,\ldots,C_nX).
\end{aligned}
\end{align}
In the context of this paper, a \textit{camera matrix} is simply a full rank matrix. The Zariski closure of the image of $\Phi_{\Ca}$ is referred to as the \textit{multiview variety}, denoted $\mathcal{M}_{\textit{\textbf{C}}}$.

A plethora of literature has explored various generalizations of linear projections of points from projective spaces, and we present a selection of notable contributions in this paragraph. Shashua and Wolf \cite{wolf2002projection} examine projections $\mathbb{P}^N \dashrightarrow \mathbb{P}^2$ for values of $N$ ranging from 3 to 6, focusing on the analysis of dynamic scenes. Work by Hartley and Vidal \cite{hartley2008perspective} models non-rigid structure and motion under perspectitve projection as projects $\PP^{3k}\dashrightarrow \PP^2$. Rydell et al. \cite{rydell2023theoretical}, in their investigation of triangulation that preserves incidence relations, consider projections $\mathbb{P}^2 \dashrightarrow \mathbb{P}^1$ and $\mathbb{P}^1 \dashrightarrow \mathbb{P}^1$, which also appear in other work \cite{quan1997affine,faugeras1998self}. In addition, $\PP^3\dashrightarrow\PP^1$ is commonly used to model radial cameras \cite{thirthala2005multi,thirthala2005radial,hruby2023four}. Studies on linear rational maps from $\mathbb{P}^N$, where $N$ is any natural number, have been conducted by Li \cite{li2018images}, and Bertolini et al. \cite{bertolini2022smooth}, who study associated critical loci, i.e. where projective reconstruction fails. Further, in this general setting, constraints for camera calibration can be obtained from Grassmann tensors as introduced by 
Hartley and Schaffalitzky \cite{hartley2009reconstruction}, whose work was recently seen in a new, more algebraic light \cite{ito2020projective}. Regarding the projection of lines, previous work \cite{faugeras1995geometry,Kileel_Thesis} studied three-view geometry. The general framework was presented in \cite{breiding2023line,breiding2023LMI}. For a line $L$ spanned by $X_0$ and $X_1$ in $\PP^3$, define the projection $C_i\cdot L$ as the line spanned by $C_iX_0$ and $C_iX_1$ in $\PP^2$. Denoting the Grassmannian of projective $k$-dimensional (projective) subspaces of $\PP^N$ by $\Gr(k,\PP^N)$, we get
\begin{align}
\begin{aligned}\label{eq: Map Line}
\Upsilon_{\Ca}: \Gr(1,\PP^3)&\dashrightarrow \Gr(1,\PP^2)^n,\\
L& \mapsto (C_1\cdot L,\ldots,C_n\cdot L).
\end{aligned}
\end{align}
The Zariski closure of the image of this map is referred to as the \textit{line multiview variety}, denoted $\mathcal{L}_{\Ca}$. 

This paper extends point and line multiview varieties by considering camera arrangements $\Ca$ consisting of $(h_i+1)\times (N+1)$ camera matrices $C_i$. In \Cref{s: GMV}, we define given a camera matrix $C$ with center $c$ a rational map $P\mapsto C\cdot P$ sending a $k$-dimensional subspace $P$ spanned by $X_0,\ldots,X_k$ to the $k$-dimensional subspace spanned by $CX_0,\ldots,CX_k$. The maps \eqref{eq: Map Point} and \eqref{eq: Map Line} naturally generalize to
\begin{align}
\begin{aligned}\label{eq: Map Joint}
\Phi_{\Ca,k}: \Gr(k,\PP^{N})&\dashrightarrow \Gr(k,\PP^{h_1})\times \cdots \times \Gr(k,\PP^{h_m}),\\
P& \mapsto (C_1\cdot P,\ldots,C_n\cdot P).
\end{aligned}
\end{align}
We call the Zariski closure of this map the \textit{(photographic) multiview variety} $\mathcal M_{\Ca,k}$. We may informally think of the multiview variety as the set of all possible simultaneous pictures taken of $k$-planes by $\Ca$. For many purposes, one can equivalently study \textit{geometric} multiview varieties, as opposed to photographic ones. These describe the geometry of the light rays, without specifying how those rays are captured by a camera into an image. A \textit{geometric camera} is a subspace $c$ in $\PP^N$, which defines for appropriate $k$ a rational map $\Gr(k,\PP^N)\dashrightarrow \Gr(\dim c+k+1,\PP^N)$ via $P\mapsto c\vee P$, where $\vee$ denotes the join of subspaces. Given a \textit{center arrangement} $\ca$ consisting of subspaces $c_i$ of $\PP^N$ with $\dim c_i\le N-k-1$, the map \eqref{eq: Map Joint} can, from the geometric point of view, be rewritten as 
\begin{align}
\begin{aligned}
\Phi_{\textit{\textbf c},k}: \Gr(k,\PP^N)&\dashrightarrow \Gr(\dim c_1+k+1,\PP^{N})\times \cdots \times \Gr(\dim c_n+k+1,\PP^{N}),\\
P& \mapsto (c_1\vee P,\ldots,c_n\vee P).
\end{aligned}
\end{align}
We call the Zariski closure of the image of $\Phi_{\textit{\textbf c},k}$ the \textit{(geometric) multiview variety} $\mathcal M_{\textit{\textbf c},k}$. If $\ca$ is the center arrangement of $\Ca$, then $\mathcal{M}_{\Ca,k}$ is linearly isomorphic to $\mathcal{M}_{\textit{\textbf c},k}$. This is stated in \Cref{le: lin isomo} and in the proof we construct an explicit isomorphism. In \Cref{s: Geo Int Joins}, we state and prove Linear Algebra results on intersections and joint of projective subspaces for the convenience of the reader. The main objective of this paper can be  formulated as the study of the \textit{fiber indices} 
\begin{align}\begin{aligned}
   \ell_{\textit{\textbf c},k}&:=\min_{(H_1,\ldots,H_n)\in \mathcal M_{\textit{\textbf c},k}} \dim H_{[n]},\\
     \upsilon_{\textit{\textbf c},k}&:=\max_{(H_1,\ldots,H_n)\in \mathcal M_{\textit{\textbf c},k}} \dim H_{[n]}.
\end{aligned}
\end{align}
In \Cref{s: Tri}, we determine $\ell_{\ca,k}$ for any camera arrangement $\ca$. In order to motivate this problem, let us first recall the problem of triangulation. Given a tuple of (noisy) data of an image tuple $\widetilde{p}=(\widetilde{p}_1,\ldots,\widetilde{p}_n)$, where $\widetilde{p}_i\in \Gr(k,\PP^{h_i})$, the \textit{triangulation problem} is to find the unique $k$-dimensional subspace $P$ that best describes the image tuple. In practice, this is done minimzing the so-called reprojection error in a choice of affine patch; see \cite{beardsley1994navigation,beardsley1997sequential,stewenius2005hard}. To do this, we find an approximation $p$ of $\widetilde{p}$ that lies on the multiview variety $\mathcal M_{\Ca,k}$, and then find a unique $k$-plane in the intersection of the back-projected planes of $p$. The algebraic complexity of triangulation has been studied in many different settings \cite{harris2018chern,EDDegree_point,rydell2023theoretical,connelly2023algebra,duff2024metric}. For the second part to be possible, $\Phi_{\Ca,k}$ must be generically injective. This occurs if and only if the dimension of $\mathcal M_{\Ca,k}$ equals the dimension of the domain, i.e. $(k+1)(N-k)$. We observe in \Cref{prop: dim} that $\dim \mathcal M_{\ca,k}=(k+1)(N-\lck)$. We therefore say that $\Mck$ is \textit{triangulable} if $\lck=k$. To state our results, we need more notation. By $c_I$, we denote the intersection of centers $c_i$ for $i\in I$. Write $\lambda\vdash n$ for a partition $\lambda$ of $[n]:=\{1,\ldots,n\}$. Define also
\begin{align}\begin{aligned}
    M_{\ca,k}(h):=\{m\in \NN^n:& \sum_{i=1}^n m_i=h, \textnormal{ and for each } I\subseteq [n],\\
    &\sum_{i\in I}m_i\le (k+1)(N-\dim c_I-k-1)\},
\end{aligned}
\end{align}
where $\NN=\ZZ_{\ge 0}$ denotes the non-negative integers. Our main results on dimensions, \Cref{cor: dim final,thm: dim other}, give the following characterizations.
\begin{introtheorem}\label{introthm: dim} Let $\ca$ be any center arrangement. Then 
    \begin{align}\begin{aligned}
    \dim \Mck=(k+1)\min_{\lambda\vdash n}\Big\{\sum_{I\in \lambda}\big(N-\dim c_I-k-1\big)\Big\}=\max_{h:\; M_{\ca,k}(h)\neq \emptyset}\;h.
\end{aligned}
\end{align}
\end{introtheorem}
If $\ca$ is the center arrangement of $\Ca$, then $\dim \MCk=\dim \Mck$. Therefore this formula also computes dimensions of photographic multiview varieties. Since determining $\dim \Mck$ and $\lck$ are equivalent problems, this theorem also computes $\lck$. The dimension formula in \Cref{introthm: dim} that uses $M_{\ca,k}$ is a generalization of the work by Li \cite{li2018images}, who stated this description in the case $k=0$. In \Cref{s: Pseudo-Dis}, we introduce \textit{pseudo-disjoint} center arrangements, whose intersections of centers do not imply constraints on the multiview variety. This property relies on specifying a value of $k$. In terms of equations, we characterize this property in \Cref{prop: Pseudo-Disjoint}.
\begin{introproposition}
The center arrangement $\ca$ is pseudo-disjoint with respect to $k$ if and only if for each $\emptyset \neq I\subseteq [n]$ such that $c_I\neq \emptyset $, we have
    \begin{align}
        \sum_{i\in I} \dim c_i\ge (|I|-1)(N-k-1) + \dim c_{I}.
    \end{align}
If $\textit{\textbf c}$ is generic, then this condition is satisfied for every $k\ge 0$.
\end{introproposition}

\Cref{introthm: dim} simplifies greatly under the assumption that $\ca$ is pseudo-disjoint with respect to $k$. Indeed, by \Cref{cor: dim final}, the dimension becomes the expected number:
\begin{introtheorem} If $\ca$ is pseudo-disjoint with respect to $k$, then
    \begin{align}
    \dim \Mck=(k+1)\min\Big\{(N-k),\sum_{i=1}^n\big(h_i-k\big)\Big\}.
\end{align}
\end{introtheorem}

In \Cref{s: Prop}, we explore proper multiview varieties. These are varieties for which calibration is possible. In symbols, $\MCk$ is \textit{proper} if
\begin{align}
    \mathcal M_{\Ca,k}\subsetneq\Gr(k,\PP^{h_1})\times \cdots \times \Gr(k,\PP^{h_n}).
\end{align}
If $\MCk$ is non-proper, then any data of an image correspondence $\widetilde{p}$ implies no conditions of $\Ca$, because set-theoretically, there are no constraints on $\mathcal M_{\Ca,k}$. This property is described as follows in \Cref{prop: Spec}.
\begin{introproposition} $\MCk$ is non-proper if and only if the center arrangement $\ca$ is pseudo-disjoint with respect to $k$ and
\begin{align}
    N-\sum_{i=1}^n(h_i-k)\ge k.
\end{align}    
\end{introproposition}
In \Cref{s: Super-Tri}, we introduce and study \textit{super-triangulability} for multiview varieties, which refers to the property that every tuple $H\in \mathcal M_{\textit{\textbf c},k}$ has that $H_i$ meet precisely in $k$ dimensions. In other words, $\uck=k$. By \Cref{prop: uck isom}, this implies that the multiview variety is isomorphic to its associated multiview blowup. This property is often used in theoretical studies of Euclidean distance degrees \cite{EDDegree_point,rydell2023theoretical}. We investigate this property for generic center arrangements $\ca$. Write
\begin{align}\begin{aligned}
  \chi_{\ca,k} &:= \frac{N-1-\sum_{i=1}^n(h_i-k)}{2},\\
    \tau_{\ca,k} &:= \Big\lfloor \chi_{\ca,k} +\sqrt{(k+1)(N-k)+(\chi_{\ca,k}-k)^2}\Big\rfloor,
\end{aligned}
\end{align}
where, $\lfloor \alpha \rfloor$ is the largest integer smaller than or equal to $\alpha$. \Cref{thm: uck} describes super-triangulability by the formula below. 
\begin{introtheorem} Let $\ca$ be a generic center arrangement. Then 
\begin{align}
     \upsilon_{\ca,k}=\min\Big\{\tau_{\ca,k},\dim c_1+k+1,\ldots,\dim c_n+k+1\Big\}. 
\end{align}
\end{introtheorem}

\bigskip

\paragraph{\textbf{Acknowledgements.}} The author was supported by the Knut and Alice Wallenberg Foundation within their WASP (Wallenberg AI, Autonomous
Systems and Software Program) AI/Math initiative. The author would like to thank Kathlén Kohn for countless helpful comments and discussions.

%%%%%%%%%%%%%%%%%%%%%%%%%%%%%%%%%%%%%%%%%%%%%%%%%%%%%%%%%%%%%%%%%%%%%%%%%%%%%%%%%%%%%%%%%%%%%%%%%%%%%%%%%%%%%%%%%%%%%%%%%%%%%%%%%%%%%%%%%%%%%%%%%%%%%%%%%%%%%%%%%%%%%%%%%%%%%%%%%%%%%%%%%%%%%%%%%%%%%%%%%%%%%%%%%%%%%%%%%%%%%

\section{Generalized Multiview Varieties}\label{s: GMV} \setcounter{equation}{0}
\numberwithin{equation}{section} This section is divided into two parts. In \Cref{ss: GeoMV}, we define geometric multiview varieties in full generality, which are the main focus of the paper. In \Cref{ss: PhoMV}, we define photographic multiview varieties in full generality, and in particular show that they are naturally linearly isomorphic to corresponding geometric multiview varieties. 

Throughout this paper, we use the following notation. Fix $N\in \NN:=\{0,1,2,\ldots\}$. A \textit{$k$-plane} is a $k$-dimensional subspace of $\PP^N$, the $N$-dimensional complex projective space. We write $\Gr(k,\PP^N)$ for the Grassmannian; the set of $k$-planes in $\PP^N$. $0$-planes are therefore points, $1$-planes are lines and $2$-planes are planes. Let $V$ and $W$ be two subspaces of $\PP^N$. $V\vee W$ denotes their \textit{join}, i.e. the subspace spanned by the vectors of $V$ and $W$, and $V\wedge W$ denotes their intersection. Let $V_1,\ldots,V_n$ be a subspaces of $\PP^N$. We adopt the notation
\begin{align}
    V_I:=\wedge_{i\in I}V_i.
\end{align}
In particular, writing $[n]=\{1,\ldots,n\}$, $V_{[n]}$ is the intersection of all $V_i$. We use the convention that $\dim \emptyset =-1$. 

\subsection{Geometric Multiview Varieties}\label{ss: GeoMV} The geometric approach to defining multiview varieties is based on back-projected planes \cite{ponce2017congruences}, referred to as \textit{lines of sight} in \cite{wolf2002projection}. Assume $N,h\in \NN$ and $N\ge h$. A full rank matrix $C$ is called a \textit{camera matrix}. A $(h+1)\times (N+1)$ camera matrix $C$ induces a rational map $C:\PP^N\dashrightarrow\PP^h$ sending $X$ to $CX$. For $k$-planes $P$ spanned by $X_0,\ldots,X_k$, we define $C\cdot P$ to be the $k$-plane spanned by $CX_0,\ldots,CX_k$ in $\PP^h$. This map is independent of the choice of spanning vectors and is well-defined precisely for $k$-planes $P$ that do not meet $c:=\ker\;C$, the \textit{center} of $C$. Therefore, throughout this paper, we assume that $\dim c\le N-k-1$, for any center $c$. Let $p$ be a $k$-plane in $\PP^h$ spanned by $x_0,\ldots,x_k$. The \textit{back-projected plane} $H$ of $p$ is the smallest subspace containing all $k$-planes in $\PP^N$ whose projection by $C$ is $p$. Denoting the Zariski closure of a set $U$ by $\overline{U}$, we give two alternative descriptions as follows. 

\begin{lemma}\label{le: back-proj} In the setting above,
\begin{align}\begin{aligned}\label{eq: back-proj}
H=&\overline{\{X\in \PP^N: X\textnormal{ does not meet }c, \textnormal{ and }  C X\in p\}}\\
=&\{X\in \PP^N: CX,x_0,\ldots,x_k \textnormal{ are linearly dependent}\}.
\end{aligned}
\end{align}
\end{lemma}
\begin{proof} Denote by \textit{(i)} the definition of $H$, \textit{(ii)} the first row of \eqref{eq: back-proj}, \textit{(iii)} the second row of \eqref{eq: back-proj}.

\textit{(i)=(ii).} Let $P$ be a $k$-plane not meeting $c$ such that $p=C\cdot P$. If $X\in P$, then $CX\in p$, showing inclusion $\subseteq$. For the other inclusion $\supseteq$, let $X$ be a point not meeting $c$. Then, since $\dim c\le N-k-1$, there is a $k$-plane $P$ containing $X$ not meeting $c$.

\textit{(ii)=(iii).} By $CX\in p$, we mean precisely that $CX$ lies in the span of $x_0,\ldots,x_k$, which shows inclusion $\subseteq$. For the other inclusion $\supseteq$, let $X\in (iii)$. If $CX\neq 0$, then $X\not\in c$, and $X\in (ii)$. If $CX=0$, meaning $X\in c$, then let $Y\in \CC^{N+1}$ be an affine representative of $X$. Since $C$ is surjective, there is a point $X_0\in \PP^N$ with $x_0=CX_0$. Let $Y_0\in \CC^{N+1}$ be an affine representative of $X_0$. If $[Z]\in \PP^N$ denotes the projectivization of $Z\in\CC^{N+1}$, then $x_0=C[\lambda Y+\mu Y_0]$ for each $\mu\neq 0$. It follows that $[\lambda Y+\mu Y_0]\in (ii)$ for each $\mu\neq 0$, and since $(ii)$ is Euclidean closed, we must have that $X=[Y]\in (ii)$, finishing the proof. 
\end{proof}

Note that the back-projected plane $H$ depends on both $p$ and $C$, but these are usually understood from the context and not specified via subscripts. If $p$ is not specified, then a back-projected plane refers to any $\dim c+k+1$-plane through $c$.  Importantly, Linear Algebra provides the following \textit{multiview identities}:

\begin{lemma}\label{le: multiview id} In the setting above,    \begin{align}\begin{aligned}\label{eq: multiid}
    \dim c&=N-h-1,\\
    \dim H&=N-h+k,\\
    \dim H&=\dim c+k+1.  
\end{aligned}
\end{align} 
\end{lemma} 

\begin{proof} Since $C$ is full rank of size $(h+1)\times (N+1)$, its kernel $c$ is of dimension $N-h-1$. We prove $\dim H=\dim c+k+1$ and note that the second equality of \eqref{eq: multiid} follows from the first and the third. By \Cref{le: back-proj}, it is clear that $H$ is non-empty. Then, by definition of $H$, we have that $H$ contains $c$ and a $k$-plane $P$ not meeting $c$ such that $C\cdot P=p$. Let $X_0,\ldots,X_k$ span $P$. If $CX_0,\ldots,CX_k$ are linearly dependent, then this means precisely that $P$ meets $c$. Therefore $CX_0,\ldots,CX_k$ span a $k$-plane, which must be equal to $p$. Then $CX$ being linearly dependent of $CX_0,\ldots,CX_k$ is equivalent to $X$ being a linear combination of $X_i$ and $c$. In other words, $H=c\vee P$. Since $c$ and $P$ are disjoint, we have $\dim c\vee P=\dim c+\dim P+1$, which is the statement.   
\end{proof}

The following statement is a direct consequence of the proof of \Cref{le: multiview id}

\begin{lemma}\label{le: CP c vee P} Let $C$ be a camera matrix with center $c$. If $P$ is a $k$-plane that does not meet $c$, then the back-projected plane of $C\cdot P$ is $c\vee P$. 
\end{lemma}

Note that the map $P\mapsto c\vee P$ is well-defined precisely for $k$-planes $P$ that do not intersect the center $c$. If a camera matrix $C$ is not specified, then a \textit{center} $c$ in this paper is any subspace of $\PP^N$. A \textit{center arrangement} $\textit{\textbf c}$ is a list $(c_1,\ldots,c_n)$ of centers in $\PP^N$, at least one in number. With \Cref{le: CP c vee P} in mind, given a center arrangement $\ca$, we define 
\begin{align}\begin{aligned}\label{eq: jointmap}
    \Phi_{\textit{\textbf c},k}: \Gr(k,\PP^N)&\dashrightarrow \Gr(\dim c_1+k+1,\PP^{N})\times \cdots\times\Gr(\dim c_n+k+1,\PP^{N}),\\
    P&\mapsto \;\:(c_1\vee P,\ldots,c_n\vee P).
    \end{aligned}
\end{align}
We denote by $\mathcal{M}_{\textit{\textbf{c}},k}$ the Zariski closure of the image $\mathrm{Im}\; \Phi_{\textit{\textbf c},k}$, and we call it the \textit{(geometric) multiview variety}.

Vital to our work are the three varieties $\Pck,\Vck$ and $\Hck$ associated to centers arrangements $\ca$ and integers $k$. These encode constraints that must hold for elements in the multiview variety $\Mck$, by which we mean that they contain $\Mck$. Below, all tuples $H=(H_1,\ldots,H_n)$ are such that $H_i\in \Gr(\dim c_i+k+1,\PP^N)$ for each $i\in [n]$.

\begin{enumerate}[label=(\roman*)]
    \item\label{enum: Pck} $\mathcal M_{\textit{\textbf c},k}\subseteq \mathcal P_{\textit{\textbf c},k}$, where 
    \begin{align}\begin{aligned}
    \mathcal{P}_{\textit{\textbf c},k}:=\{H:c_i\subseteq  H_i \textnormal{ for all } i\in [n]\}.
\end{aligned}
\end{align}
This inclusion holds, because by construction, every tuple $H$ in the image of $\Phi_{\ca,k}$ is of the form $H_i=c_i\vee P$ for some $k$-plane $P$ not meeting any center. Then $c_i\subseteq H_i$ for each $i$ also in the closure.

\item\label{enum: Hck} $\mathcal M_{\textit{\textbf c},k}\subseteq \mathcal H_{\textit{\textbf c},k}$, where 
\begin{align}\begin{aligned}
    \calH_{\textit{\textbf c},k}:=\{H:\dim H_{[n]}\ge k \}.
\end{aligned}
\end{align}
By the same reasoning as above, for $H\in \mathrm{Im}\; \Phi_{\textit{\textit{\textbf c}},k}$ there exists a $k$-plane $P$ contained in each $H_i$. $\calH_{\textit{\textbf c},k}$ generalizes the concurrent lines variety from \cite{ponce2017congruences}.  

\item\label{enum: Vck} $\mathcal M_{\textit{\textbf c},k}\subseteq \mathcal V_{\textit{\textbf c},k}$, where \begin{align}\begin{aligned}
    \calV_{\textit{\textbf c},k}:=\{H:\dim H_I\ge \dim c_I+k+1 \textnormal{ for all }I\subseteq [n] \textnormal{ with }|I|>1 \textnormal{ and }c_I\neq \emptyset \}.
\end{aligned}
\end{align}
These equations arise from the intersections of centers. To see this, let $P$ be a $k$-plane that meets no centers. If $c_I\neq \emptyset$ for some $I\subseteq [n]$, then each back-projected plane $c_i\vee P,i\in I,$ contains the $\dim c_I+k+1$-plane $c_I\vee P$. Therefore, given $H\in \mathrm{Im}\;\Phi_{\textit{\textbf c},k}$, we have that $\dim H_I\ge\dim c_I+k+1$.
\end{enumerate}
In particular, 
\begin{align}
    \Mck\subseteq \Pck\cap \Hck\cap \Vck.
\end{align}
For $k=0$, this inclusion is in many cases known to be an equality. For $k=1$ it is not always an equality, which was a main insight of \cite{breiding2023line}. In \Cref{s: Prop}, we give a condition for when $\Mck$ is an irreducible component of the right-hand side.

%%%%%%%%%%%%%%%%%%%%%%%%%%%%%%%%%%%%%%%%%%%%%%%%%%%%%%%%%%%%%%%%%%%%%%%%%%%%%%%%%%%%%%%%%%%%%%%%%%%%%%%%%%%%%%%%%%%%%%%%%%%%%%%%%

\subsection{Photographic Multiview Varieties}\label{ss: PhoMV} This paper focuses on geometric multiview varieties, because they are more convenient to work with. Nevertheless, it is important to explain their relation to photographic multiview varieties in detail, which we do in \Cref{le: lin isomo}. This is because photographic multiview varieties appear in Computer Vision applications.

We define a \textit{camera arrangement} to be a list $\Ca=(C_1,\ldots,C_n)$ of camera matrices $C_i$ of sizes $(h_i+1)\times (N+1)$, at least one in number. A camera arrangement $\Ca$ defines a \textit{joint image} map as follows:
\begin{align}\begin{aligned}
    \Phi_{\Ca,k}: \Gr(k,\PP^N)&\dashrightarrow \Gr(k,\PP^{h_1})\times \cdots\times\Gr(k,\PP^{h_n}),\\
    P&\mapsto \;\:(C_1\cdot P,\ldots,C_n\cdot P).
\end{aligned}
\end{align}
The \textit{(photographic) multiview variety} $\mathcal{M}_{\Ca,k}$ is the Zariski closure of the image of $\Phi_{\Ca,k}$. 

We note that as an analogue to \Cref{ss: GeoMV}, there are varieties $ \mathcal H_{\Ca,k}$ and $ \mathcal V_{\Ca,k} $ that encode constraints that must hold for elements in $\calM_{\Ca,k}$ using \Cref{le: CP c vee P}. Here, $ \mathcal H_{\Ca,k}$ is the set of image tuples $p$ whose tuple of back-projected planes $H$ lies in $\Hck$, and $ \mathcal V_{\Ca,k} $ is defined analogously. The next result motivates our subsequent focus on geometric multiview varieties. 

\begin{lemma}\label{le: lin isomo} If $\ca$ is the center arrangement of $\Ca$, then $\mathcal{M}_{\Ca,k}$ is linearly isomorphic to $\mathcal{M}_{\textit{\textbf c},k}$.
\end{lemma}

The explicit linear isomorphism from the statement is described in the proof. This isomorphism allows us to express the set-theoretic equations defining the photographic multiview varieties in terms of the geometric ones.

\begin{proof} Let $C:\PP^N\dashrightarrow\PP^h$ be a camera matrix. It suffices to show that the map $p\mapsto H$, sending a $k$-plane in $\PP^h$ to its back-projected plane in $\PP^N$ is a linear isomorphism. Below we explicitely describe $p\mapsto H$ and its inverse as linear maps.  %This is because then $\mathrm{Im}\; \Phi_{\Ca,k}$ and $\mathrm{Im}\; \Phi_{\textit{\textbf c},k}$ are linearly isomorphic via the map that coordinateswise sends an image $k$-plane $p_i$ to its back-projected plane $H_i$. Then, since $\mathcal{M}_{\Ca,k}$ and $\mathcal{M}_{\textit{\textbf c},k}$ are the Euclidean closures of $\mathrm{Im}\; \Phi_{\Ca,k}$ and $\mathrm{Im}\; \Phi_{\textit{\textbf c},k}$, respectively, the limit points of the closures must also be isomorphic via the same linear map.

Recall that we assume $h\le N$. Denote by $C^\dagger$ any choice of $(N+1)\times (h+1)$ right pseudo-inverse of $C$, meaning a matrix such that $CC^\dagger$ is the $(h+1)\times (h+1)$ identity matrix. The linear map 
\begin{align}\begin{aligned}
    \Gr(k,\PP^h)\to \Gr(k,\PP^N), \quad  p\mapsto  \wedge^{k+1}C^\dagger p, 
\end{aligned}
\end{align}
sends $p$ to a $k$-plane in $\PP^N$ that lies in the image of $C^\dagger$, and is therefore disjoint from the center $c$ by the fact that $\mathrm{Im}\; C^\dagger $ affinely intersects $\mathrm{ker}\; C$ only in the point $0$. Observe that $\wedge^{k+1}C^\dagger$ is a linear surjection from $\Gr(k,\PP^N)$ to the $k$-planes in $\PP^N$ that lie in $\mathrm{Im}\; C^\dagger$, since $C^\dagger$ is full rank and $\dim \mathrm{Im}\; C^\dagger=h$. There is a matrix $\hat{C}$ such that
\begin{align}\begin{aligned}
\Gr(k,\PP^h)\to \Gr(\dim c+k+1,\PP^N), \quad   p \mapsto \hat{C}p= c\vee (\wedge^{k+1}C^\dagger p).
\end{aligned}
\end{align}
By the multiview identities, $\hat{C}$ is a size
\begin{align}\label{eq: A}
   {N+1\choose N-\dim c-k}\times {N-\dim c\choose N-\dim c-k-1}
\end{align}
matrix. In particular, it has at least as many rows as columns. We claim that $\hat{C}$ is full rank and that therefore the left pseudo-inverse $\hat{C}^\dagger$ is its inverse morphism. 

One can construct a full rank matrix $A$ of size $ {h+1\choose k+1}\times {N+1\choose \dim c+k+1}$ that is injective on $\mathrm{Im}\;\hat{C}$. Since both $A$ and $\hat{C}$ are injective, their product is injective from $\Gr(k,\PP^{h})$ to itself. This means that the dimension of a fiber of $A\hat{C}$ is 0-dimensional, and by the fiber dimension theorem \cite[Chapter 1, \S 8, Corollary 1]{mumford1999red}, $A\hat{C}$ is dominant onto its image. Then $A\hat{C}$ must be invertible, as $\Gr(k,\PP^{h})$ spans $\PP^{{h+1\choose k+1}-1}$, and $\hat{C}$ must be full rank.
\end{proof}

From the applied point of the view, the image of $\Phi_{\Ca,k}$ is in some sense more important than its closure, the multiview variety. The next proposition describes the image in terms of $\mathcal M_{\Ca,k}$, generalizing a result of \cite{trager2015joint}.

\begin{proposition}\label{prop: image} The image of $\Phi_{\Ca,k}$ is described via the multiview variety as
\begin{align}
    \mathrm{Im}\,\Phi_{\Ca,k}= \calM_{\Ca,k}\setminus \mathcal{T}, 
\end{align} 
where
\begin{align}\mathcal{T}:=\Big\{p\in \calM_{\Ca,k}: \textnormal{ there exists }i \textnormal{ such that }\dim c_i\wedge H_{[n]} +k\ge \dim H_{[n]}\Big\}.\end{align}
\end{proposition}

\begin{proof} Let $p=(p_1,\ldots,p_n)\in \mathcal M_{\Ca,k}$ and write $(H_1,\ldots,H_n)$ for its tuple of back-projected planes. $p$ lies in the image of $\Phi_{\Ca,k}$ if and only if there is a $k$-plane $P$ inside the intersection $H_{[n]}$ meeting no center. This happens if and only if each center inside $H_{[n]}$ satisfies 
\begin{align}
    \dim c_i\wedge H_{[n]}+k+1\le \dim H_{[n]}, 
\end{align}
which is the statement.
\end{proof}

\begin{remark} There are two natural extensions of the setting described in this section. We could allow for matrices $C:\PP^N\dashrightarrow\PP^h$ that are not full rank and we could allow for $h>N$. However, the image of a linear map that is not full rank is isomorphic to the image of the linear map we get by removing a minimal number of rows of the corresponding matrix, making it full rank. We have a similar situation if $h>N$. In either case, for the purposes of this paper, no interesting geometry arises from these extensions.% then the linear  is an isomorphism onto its image and if $C$ is not full rank we may replace it with a full rank matrix of smaller size with the same kernel by removing rows such that their images are isomorphic. Then, since $\mathcal{M}_{\textbf c,k}$ depends only on the centers of $\textbf c$, we can without restriction disregards these extensions.  
\end{remark}

\section{Geometry of Intersections and Joins}\label{s: Geo Int Joins} For the sake of completeness, we include basic results on the intersections and joins of subspaces of projective spaces that are helpful throughout the article. As we saw in \Cref{ss: GeoMV},
\begin{align}
    \Mck\subseteq \Pck\cap \Hck\cap \Vck,
\end{align}
and both $\Vck$ and $\Hck$ are defined via intersections of subspaces, which motivates this study. Note that the proofs we provide are nothing but exercises in Linear Algebra. The reader may skip this section at first and only come back to it whenever it is referenced in proofs in later sections.

To begin, we state the relation between the dimensions of intersections and joins. 

\begin{lemma}\label{le: int + join} Let $H_1,H_2$ be subspaces of $\PP^N$. We have
    \begin{align}
        \dim H_1\wedge H_2+\dim H_1\vee H_2=\dim H_1+\dim H_2.    \end{align}
\end{lemma}
\begin{proof} Let $v_0,\ldots,v_r\in \CC^{N+1}$ be linearly indepent vectors that span $H_1$ and $w_0,\ldots,w_s\in \CC^{N+1}$ linearly independent vectors that span $H_2$. Consider the matrix
\begin{align}
    M=\begin{bmatrix}
        v_0& \ldots & v_r & w_0 & \ldots & w_s
    \end{bmatrix}.
\end{align}
By the dimensions theorem in Linear Algebra, $\rank M+\dim \ker M=r+s+2$, which equals $\dim H_1+\dim H_2+2$. It is clear that $\rank M=\dim H_1\vee H_2+1$. We are done if we can show that $\dim \ker M =\dim H_1\wedge H_2+1$. To do this, we first write $M=\begin{bmatrix}M_v & M_w
\end{bmatrix}$, where $M_v$ is the matrix of column vectors $v_i$ and $M_w$ is the matrix of column vectors $w_i$. Similarly, given $\lambda\in \CC^{r+s+2}$, write $\lambda=[\lambda_v;\; -\lambda_w]$ for the column vector comprised of $\lambda_v\in \CC^{r+1}$ and $\lambda_w\in \CC^{s+1}$. We have that $X\in H_1\wedge H_2$ if and only if there exists a $\lambda$ such that $Y=M_v\lambda_v=M_w\lambda_w$, where $Y\in \CC^{N+1}$ is an affine representative of $X\in \PP^N$. However, by construction, $M_v\lambda_v=M_w\lambda_w$ if and only if $\lambda\in \ker M$. Further, each $\lambda\in \ker M$ gives a unique element $M_v\lambda_v=M_w\lambda_w$ (up to scaling) of $H_1\wedge H_2$ by the fact that $M_v$ and $M_w$ are full rank. This shows that the cone over $H_1\wedge H_2$ is equals the kernel of $M$, viewed as a projective space.
\end{proof}

The next lemma gives a condition that guarantees that a tuple of subspaces meet in a specificed dimension. 

\begin{lemma}\label{le: intDim} Let $H_i$ be subspaces of $\PP^N$ such that  
\begin{align}
  \sum_{i=1}^n \dim H_i\ge (n-1)N+k.  
\end{align}
Then $\dim H_{[n]}\ge k$.\end{lemma}

\begin{proof} The statement clearly holds for $n=1$. Further, $\dim H_1+\dim H_2\ge N+k$ implies by \Cref{le: int + join} that $\dim H_1\wedge H_2\ge k$, since $\dim H_1\wedge H_2\le N$. Now assume by induction that the statement holds for some $n\ge 2$. Let $\dim H_{n+1}=N-j$ for some integer $j$ and
\begin{align}
  \dim H_1+\cdots +\dim H_{n+1}  \ge nN+k.  
\end{align}
We can rewrite this inequality to
\begin{align}
  \dim H_1+\cdots +\dim H_{n}  \ge (n-1)N+k+j.  
\end{align}
By induction, $\dim H_{[n]}\ge k+j$ and it follows that 
\begin{align}
    \dim H_{[n]}+\dim H_{n+1}\ge N+k,
\end{align}
and we are done by the base case $n=2$. 
\end{proof}

We give provide a statement on the manipulation of intersections with joins, and note that it does not hold without the assumption that $V\subseteq W$.

\begin{lemma}\label{le: uvw} Let $U,V$ and $W$ be subspaces of $\PP^N$. If $V\subseteq W$, then 
\begin{align}
    W\wedge (V\vee U)=V\vee (W\wedge U).
\end{align}
\end{lemma}

\begin{proof} We perform the proof by identifying $U,V$ and $W$ with their affine cones in $\CC^{N+1}$. 

$\subseteq)$ Let $x\in W\wedge (V\vee U)$. Then $x\in W$ and $x\in V\vee U$, and we can write $x=\lambda v+ \mu u$ for $v\in V, u\in U$ and two constants $\lambda,\mu$. If $\mu=0$, then $x\in V$ and in particular $x\in V\vee (W\wedge U) $. Otherwise, recall that $v\in W$, since $V\subseteq W$. This implies that $u=(x-\lambda v)/\mu\in W$. Then by definition, $x=\lambda v +\mu u\in V\vee (W\wedge U) $.

$\supseteq)$ Take $x\in V\vee (W\wedge U)$. Then $x=\lambda v+\mu u$, where $v\in V\subseteq W$ and $u\in W\wedge U$. This implies that $x\in W$ and $x\in V\vee U$. 
\end{proof}

For the next result we introduce new notation. Given a subspace $V$ of $\PP^N$, we let a ``dual'' $V^*$ denote a subspace of dimension $N-\dim V-1$ that is disjoint from $V$.

\begin{lemma}\label{le: U vee V*} Let $V\subseteq W$ be an inclusion subspaces of $\PP^N$. Then 
\begin{align}
    \dim W\wedge V^*=\dim W-\dim V-1.
\end{align}
\end{lemma}
\begin{proof} First note that $\PP^N=V\vee V^*\subseteq W\vee V^*$, implying that $W\vee V^*=\PP^N$. Then by \Cref{le: int + join},
\begin{align}
    \dim W\wedge V^*=\dim W+\dim V^*-\dim W\vee V^*,
\end{align}
which simplifies to the statement.
\end{proof}

The next lemma is a simple observation that we use to work inside a space of lower dimension; allowing for induction arguments over $N$. We leave the straight-forward proof to the reader.

\begin{lemma}\label{le: isomo onto *} Let $c_0\subseteq c$ be an inclusion of subspaces of $\PP^N$, and let $d$ be an integer $\dim c\le d\le N$. There is an isomorphism between 
\begin{enumerate}
    \item $d$-planes $H$ through $c$, and 
    \item $(d-\dim c_0-1)$-planes inside $c_0^*$ through $c':=c\wedge c_0^*$,
\end{enumerate}
via the map $H\mapsto H':=H\wedge c_0^*$, with inverse $H'\mapsto H=c_0\vee H'$.
\end{lemma}

%\begin{proof} Let $H$ be a $d$-plane through $c$. By \Cref{le: uvw}, we have $H=c_{0}\vee (H\wedge c_{0}^*)=c_0\vee H'$. Similarly, $c=c_0\vee(c\wedge c_0^*)=c_0\vee c'$. This shows that $H$ is uniquely determined by its intersection $H'$ with $c_0^*$. It follows by \Cref{le: int + join} that $\dim H'$ is $d-\dim c_0-1$. This shows that the map is an isomorphism onto its image. Next, take $G$ to be any $d-\dim c_0-1$-plane in $c_0^*$ that contains $c'$. Then $c_0\vee G$ is a $d$-plane in $\PP^N$ that contains $c$, since it contains $c_0\vee c'$ and $c_0\vee G$ is of dimension $d$ by \Cref{le: int + join}.
%\end{proof}

Intersections and joins behave as expected with respect to this isomorphism of \Cref{le: isomo onto *}:

\begin{lemma}\label{le: int join *}
Let $\textit{\textbf c}$ be a center arrangement. Let $H_i$ denote subspaces through $c_i$, and write $H_i':=H_i\wedge c_{[n]}^*$. For a set of indices $I,J\subseteq [n]$, we have
\begin{align}
\begin{aligned}\label{eq: int and join}
   H_I\wedge H_J&=c_{[n]}\vee (H_I'\wedge H_J'),\\
   H_I\vee H_J&= c_{[n]}\vee (H_I'\vee H_J').   
\end{aligned}
\end{align}
\end{lemma}

\begin{proof} The left-hand sides of \eqref{eq: int and join} contain $c_{[n]}$ and $H_I'\wedge H_J'$ and $H_I'\vee H_j'$, respectively, which span the right-hand sides.

For the other inclusions, we start with the first equality. We can write $H_I\wedge H_J$ as $c_{[n]}\vee (H_{I\cup J}\wedge c_{[n]}^*)$ by \Cref{le: uvw}. It can be readily checked that $(H_{I\cup J}\wedge c_{[n]}^*)$ equals $ H_I'\wedge H_J'$. For the second equation, it suffices to check that the dimensions of $H_I\vee H_J$ and $c_{[n]}^*\vee (H_I'\vee H_J')$ are the same, but this follows from the first equality using \Cref{le: int + join}.
\end{proof}

Our final result from this section tell us the dimension of intersections of generic subspaces through fixed centers.

\begin{lemma}\label{le: subsp dim formula} Let $c_1\subseteq H_1$ and $c_2\subseteq H_2$ be inclusions of subspaces of $\PP^N$, where $c_i$ are fixed and $H_i$ generic through $c_i$, respectively. Then 
   \begin{align}\label{eq: H1 H2 int gen}
    \dim H_1 \wedge H_2= \max \{ \dim c_1 \wedge c_2,\dim H_1+ \dim H_2-N \}.
\end{align}  
\end{lemma}

\begin{proof} Assume first that $c_1\wedge c_2=\emptyset$. By linear transformations, we can take any two disjoint subspaces to any other two disjoint subspaces of the same dimensions. Therefore we can take $c_1$ and $c_2$ to be generic, and in particular consider $H_1$ and $H_2$ to be generic in $\PP^N$. Generic $H_1$ and $H_2$ correspond to choices of generic vectors $v_0,\ldots, v_r$ and  $w_0,\ldots,w_s$, where $r=\dim H_1$ and $s=\dim H_2$. If $s+r+2\ge N+1$, then $\dim H_1\vee H_2=N$, and otherwise $\dim H_1\wedge H_2=\emptyset$. In the former case, \Cref{le: int + join} gives us the statement. In the latter case, we have $\dim H_1+\dim H_2<N$ and therefore $\dim H_1\wedge H_2=-1$ satisfies \eqref{eq: H1 H2 int gen} even in this case.  

Assume second that $c_1\wedge c_2\neq \emptyset$. We use \Cref{le: isomo onto *}, namely, we write $H_i':=H_i\wedge (c_1\wedge c_2)^*$. If $H_i$ is generic through $c_i$, then $H_i'$ is generic through $c_i':=c_i\wedge (c_1\wedge c_2)^*$ inside $(c_1\wedge c_2)^*$. Observe that $c_i'$ are disjoint: $c_1'\wedge c_2'= (c_1\wedge c_2)\wedge (c_1\wedge c_2)^*=\emptyset$. Then, by the first part of the proof and the fact that $\dim H_i'=\dim H_i-\dim c_1\wedge c_2-1$, 
\begin{align}\begin{aligned}\label{eq: in *}
    \dim H_1'\wedge H_2' &=\max \{ -1,\dim H_1'+ \dim H_2'-\dim (c_1\wedge c_2)^* \}\\
    &=\max \{ -1,\dim H_1+ \dim H_2-N-\dim c_1\wedge c_2-1 \}.
\end{aligned}
\end{align}
Then by \Cref{le: int join *}, $\dim H_1\wedge H_2= \dim H_1'\wedge H_2'+\dim c_1\wedge c_2+1$. Plugging this into \eqref{eq: in *}, we are done.
\end{proof}

%%%%%%%%%%%%%%%%%%%%%%%%%%%%%%%%%%%%%%%%%%%%%%%%%%%%%%%%%%%%%%%%%%%%%%%%%%%%%%%%%%%%%%%%%%%%%%%%%%%%%%%%%%%%%%%%%%%%%%%%%%%%%%%%%%%%%%%%%%%%%%%%%%%%%%%%%%%%%%%%%%%%%%%%%%%%%%%%%%%%%%%%%%%%%%%%%%%%%%%%%%5

%%%%%%%%%%%%%%%%%%%%%%%%%%%%%%%%%%%%%%%%%%%%%%%%%%%%%%%%%%%%%%%%%%%%%%%%%%%%%%%%%%%%%%%%%%%%%%%%%%%%%%%%%%%%%%%%%%%%%%%%%%%%%%%%%%%%%%%%%%%%%%%%%%%%%%%%%%%%%%%%%%%%%%%%%%%%%%%%%%%%%%%%%%%%%%%%%%%%%%%%%%5

\section{Triangulability}\label{s: Tri} 
Triangulation in the context of Computer Vision should not be confused with triangulation of topological spaces. Instead, it refers to reconstructing a world feature based on (noisy) data of feature matches, as described in the introduction. In this paper, we study which multiview varieties can be used for triangulation. The requirement is that the projection map $\Phick$ is generically injective, or equivalently that it is dominant onto its image. We give two distinct characterization of this property. In the process, we describe the dimensions of multiview varieties.
%%%%%%%%%%%%%%%%%%%%%%%%%%%%%%%%%%%%%%%%%%%%%%%%%%%%%%%%%%%%%%%%%%%%%%%%%%%%%%%%%%%%%%%%%%%%%%%%%%%%%%%%%%%%%%%%%%%%%%%%%%%%%%%%%%%%%%%%%%%%%%%%%%%%%%%%%%%%%%%%%%%%%%%%%%%%%%%%%%%%%%%%%%%%%%%%%%%%%%%%%%5

First we need new notation. For a center arrangement $\textit{\textbf c}$ and a non-negative integer $k$, define 
\begin{align}
   \ell_{\ca,k}&:=\min_{(H_1,\ldots,H_n)\in \mathcal M_{\ca,k}} \dim H_{[n]}
\end{align}
The value $\lck$ has the following natural description.

\begin{lemma}\label{le: semi-upper} $\ell_{\textbf c,k}$ is the dimension of 
\begin{align}\label{eq: wedge c_i vee P}
  (c_1\vee P)\wedge \cdots \wedge (c_n\vee P)
\end{align}
%\begin{align}
%  \bigwedge_{i=1}^n c_i\vee P
%\end{align}
for generic $P\in \Gr(k,\PP^N)$.
\end{lemma}

\begin{proof} Let $\mathcal U$ denote the set of $P\in \Gr(k,\PP^N)$ meeting no centers such that
\begin{align} \label{eq: intdim}
 \dim \bigwedge_{i=1}^n c_i\vee P\ge \ell_{\ca,k}+1.
\end{align}
This is a variety by the upper semi-continuity of dimensions \cite[Chapter 1, \S 8, Corollary 3]{mumford1999red}. If $\mathcal U= \Gr(k,\PP^N)$, then any $H\in \mathrm{Im}\; \Phick$ would have $\dim H_{[n]}\ge \lck+1$. This inequality would also hold in the closure $H\in \Mck$, which is a contradiction to the definition of $\lck$. Therefore, $\mathcal U$ is a proper subvariety of $\Gr(k,\PP^N)$ and for generic $P\in \Gr(k,\PP^N)$, we have that the dimension of \eqref{eq: wedge c_i vee P} is $\lck$.
\end{proof}

Triangulation is impossible if $\lck>k$, for the same reason that you can't reconstruct a point $X$ in $\PP^3$ based on one single projection $x$ of $X$ onto a camera plane $\PP^2$; there is a 1-dimensional family of points in the preimage, i.e. the back-projected line $H$, and any point in $H$ projects onto $x$. However, if $\lck=k$, then a generic element of the multiview variety can be triangulated. Indeed, since $P$ lies in \eqref{eq: wedge c_i vee P}, if $\lck=k$, then $P=H_{[n]}$, where $H_i=c_i\vee P$ are the back-projected planes. The property $\lck=k$ is also equivalent to $\Mck$ being birationally equivalent to $\Gr(k,\PP^n)$; the inverse rational map to $\Phick$ sends $(H_1,\ldots,H_n)$ to $P=H_{[n]}$. We have now motivated the titular definition of this section.

\begin{definition} We say that $\Mck$ is:
\begin{itemize}
    \item \textit{Triangulable} if a generic tuple $H\in \mathcal M_{\textit{\textbf c},k}$ can be triangulated, meaning $\lck=k$. 
\end{itemize}
\end{definition}

We have seen that $\lck=k$ is a desirable property with a natural interpretation. Another interpretation of $\lck$ is through the dimension of multiview varieties. 

\begin{proposition}\label{prop: dim} The multiview variety $\mathcal{M}_{\textbf c,k}$ is irreducible of dimension 
\begin{align}
  (k+1)(N-\ell_{\textbf c,k}).
\end{align}
%The real dimensions is the same if the cameras are real.
\end{proposition}

\begin{proof} The dimension of the set of $k$-planes in a given $\lck$-plane is $(k+1)(\lck-k)$. Next, take a generic $k$-plane $P$ and let $H_i=c_i\vee P$. Then $\dim H_{[n]}=\lck$ by \Cref{le: semi-upper} and $P\in H_{[n]}$. Since $P$ does not meet any center, it follows that a generic $k$-plane in $H_{[n]}$ does not meet any center. Then the generic fiber of $\Phick$ is $(k+1)(\lck-k)$-dimensional. Since $\Phick$ is dominant onto $\Mck$, the fiber dimension theorem says
\begin{align}
 \dim \Mck+ (k+1)(\lck-k)=(k+1)(N-k),
\end{align}
and the statement follows.
\end{proof}

\begin{example} Consider a linear projection $\PP^3\dashrightarrow \PP^1\times \PP^1$, given by camera matrices with distinct line centers $c_1$ and $c_2$. Then $\mathcal{M}_{\textit{\textbf c},0}$ is not triangulable, because any two back-projected planes meet in this case exactly in a line, showing $\ell_{\ca,0}=1$.\hspace*{\fill}$\diamondsuit\,$ \end{example}

\begin{example} Consider a linear projection $\mathrm{Gr}(1,\PP^3)\dashrightarrow \Gr(1,\PP^2)^2$, given by camera matrices with distinct point centers $c_1$ and $c_2$. Then $\mathcal{M}_{\textit{\textbf c},1}$ is triangulable, because two back-projected planes generically meet in exactly a line, implying $\ell_{\ca,1}=1$. \hspace*{\fill}$\diamondsuit\,$ \end{example}

The main results of this section give a formulae for $\lck$ and in turn the dimension of $\Mck$. Before we state them, we provide two lemmas that furnish an induction argument on the number of cameras. In particular, we show the statement for center arrangements of two cameras in \Cref{le: lck n = 2}. The proof of this lemma is a variant of the proof of \Cref{le: subsp dim formula}. Recall that $h_i=N-\dim c_i-1$ by the multiview identities.

\begin{lemma}\label{le: lck n = 2} Let $\textit{\textbf c}=(c_1,c_2)$ be a center arrangement. Then
\begin{align}
    \lck = \max\{\dim c_1\wedge c_2+k+1,\dim c_1+\dim c_2+2k+2-N\}.
\end{align}
\end{lemma}

\begin{proof} Assume $c_1\wedge c_2=\emptyset$. Then, since $c_1,c_2$ are disjoint subspaces, linear transformations can send them to any any two other disjoint subspaces of the same dimension. We therefore may treat $c_1$ and $c_2$ as if they are generic subspaces. Let $H_1,H_2$ be generic subspaces of dimension $\dim c_i+k+1$ in $\PP^N$. They meet in $\max\{-1,\dim H_1+\dim H_2-N\}$ dimensions by \Cref{le: subsp dim formula}. If $\dim H_1+\dim H_2\ge N+k$, then $\dim H_1\wedge H_2\ge k$. Let $c_i$ be generic $\dim c_i$-planes inside $H_i$. These $c_i$ are disjoint, because $H_i$ are generic. Next, fix a $k$-plane $P$ inside $H_1\wedge H_2$. Generic $c_i$ inside $H_i$ don't meet $P$, since $\dim H_i=\dim c_i+k+1$. We have shown that for this choice of $c_i$, $\lck=\dim H_1+\dim H_2-N$. If $\dim H_1+\dim H_2<N+ k $, then generic $H_i$-planes meet in less than $k$ dimensions. We can choose such planes that meet in exactly $k$ dimensions, and let this $k$-plane be $P$. Generic $c_i$ in $H_i$ don't meet $P$. This shows that $\lck =k$.  

Assume $c_1\wedge c_2\neq \emptyset$. As in the proof of \Cref{le: subsp dim formula}, we consider the arrangement $\textit{\textbf c}'$ consisting of $c_i'=c_i\wedge (c_1\wedge c_2)^*$, which is a disjoint arrangement in $(c_1\wedge c_2)^*$. As in \Cref{le: isomo onto *}, the set of back-projected planes are isomorphic to the set of $(\dim H_i-\dim (c_1\wedge c_2)^*)-1$-planes inside $(c_1\wedge c_2)^*$ through $c_i'$, respectively. By the first part of the proof,
\begin{align}\begin{aligned}
    \ell_{\textit{\textbf c}',k} &=\max\{k,\dim c_1'+\dim c_2'+2k+2-\dim (c_1\wedge c_2)^*\}\\
    &=\max\{k,\dim c_1+\dim c_2+2k+1-N-\dim c_1\wedge c_2\}.
\end{aligned}
\end{align}
We therefore aim to prove that $\lck=\ell_{\textit{\textbf c}',k}+\dim c_1\wedge c_2+1$. Take a generic $k$-plane $P$ in $\PP^N$. It is disjoint from every center. Moreover, it is contained in a $(N-\dim c_1\wedge c_2-1)$-plane away from $c_1\wedge c_2$. We may take $(c_1\wedge c_2)^*$ to be this subspace. Write $H_i=c_i\vee P$. Then $\lck=H_1\wedge H_2$ by \Cref{le: semi-upper} and the rest follows from \Cref{le: int join *}.
\end{proof}

\begin{lemma}\label{le: lck lcprimk} Let $\ca$ be a center arrangement of $n\ge 2$ centers. Fix a generic $\dim c_n+k+1$-plane $H_n$ through $c_n.$ Define 
\begin{align}
    \ca':=(c_1\wedge H_n,\ldots,c_{n-1}\wedge H_n), 
\end{align}
viewed as a center arrangement inside $H_n\cong \PP^{\dim c_n+k+1}$. Then
\begin{align}
    \lck=\ell_{\ca',k}.
\end{align}    
\end{lemma}

\begin{proof} By \Cref{le: semi-upper}, for a generic $k$-plane $P$ in $\PP^N$, the back-projected planes $H_i=c_i\vee P$ meet in exactly an $\lck$-plane. This intersection lies inside each $H_i$. Fix a generic $\dim c_n+k+1$-plane $H_n$ through $c_n.$ Take a generic $k$-plane $P$ inside $H_n$. Since $H_n$ is generic, $c_i\vee P$ meet in exactly an $\lck$-plane. Observe that
\begin{align}\label{eq: lck compare}
    \bigwedge_{i=1}^{n-1} \big((c_i\wedge H_n)\vee P\big)= \bigwedge_{i=1}^{n-1} \big((c_i\vee P)\wedge H_n\big),
\end{align}
because $P\subseteq  H_n$ and by \Cref{le: uvw}, we have $(c_i\wedge H_n)\vee P= (c_i\vee P)\wedge H_n$ for each $i$. We are done, since the dimension of the left-hand side of \eqref{eq: lck compare} is $\ell_{\textit{\textbf c}',k}$ and the dimension of the right-hand side is $\lck$. 
\end{proof}

Denote by $\lambda\vdash n$ a partition of $[n]$.

\begin{theorem}\label{thm: lck final} We have
    \begin{align}\label{eq: lck part}
        \lck=N+\max_{\lambda \vdash n}\Big\{\sum_{I\in \lambda} \big(\dim c_I+k+1-N\big)\Big\}.
    \end{align}
\end{theorem}

We now have a complete characterization for when $\Mck$ is triangulable, namely when the right-hand side of \eqref{eq: lck part} equals $k$. An alternative form of the right-hand side of \eqref{eq: lck part} is
 \begin{align}\label{eq: lck part alt}
        N+\max_{\lambda \vdash n}\Big\{-|\lambda|(N-k-1)+\sum_{I\in \lambda} \dim c_I\Big\}.
    \end{align}

\begin{proof} One can check that the statement holds for center arrangements consisting of only one center $c$. Namely, in this case the only partition consists of the element $I=\{1\}$. Then \eqref{eq: lck part} says that $\lck=N-(N-k-1)+\dim c =\dim c+k+1$, which holds. We proceed by assuming that $\ca$ consists of at least two centers, and argue via induction.

We fix a generic back-projected plane $H_n$. We define $\ca'$ via $c_i':=c_i\wedge H_n$ for $i=1,\ldots,n-1$ and use \Cref{le: lck lcprimk}. Assuming via induction that the statement holds for $n-1$ centers, we have
\begin{align}\label{eq: lcprimk}
    \lck=\ell_{\ca',k}=\dim c_n+k+1+\max_{\lambda \vdash n-1}\Big\{\sum_{I\in \lambda}\big(\dim c_I\wedge H_n-\dim c_n\big)\Big\}.
\end{align}
Next, we write $\widehat{\ell}$ for the right-hand side of \eqref{eq: lck part}. We aim to show $\lck=\widehat{\ell}$. For this, it follows from \Cref{le: subsp dim formula} that for any $I\subseteq [n]$, 
\begin{align}\label{eq: use this}
     \dim c_I+k+1-N\le \dim c_I\wedge H_n-\dim c_n.
\end{align}
  
  For a partition $\lambda$ of $[n]$, there is a exactly one $I\in \lambda$ with $n\in I$, and we denote this $I$ by $I_n$. Write $\widehat{\ell}_1$ for \eqref{eq: lck part}, where the maximum is over partitions with $I_n=\{n\}$, and write $\widehat{\ell}_2$ for \eqref{eq: lck part}, where the maximum is over all other partitions. By construction, $\widehat{\ell}=\max\{\widehat{\ell}_1,\widehat{\ell}_2\}$. By \eqref{eq: use this},
\begin{align}\begin{aligned}
    \widehat{\ell}_1&= N+ \max_{\lambda \vdash n-1}\Big\{\big(\dim c_n+k+1-N\big)+\sum_{I\in \lambda} \big(\dim c_I+k+1-N\big)\Big\}\\
    &\le \dim c_n+k+1+\max_{\lambda \vdash n-1}\Big\{\sum_{I\in \lambda} \big(\dim c_I\wedge H_n-\dim c_n\big)\Big\}.
\end{aligned}
\end{align}
The right-hand side of this equations equals $\lck$. In general, we have $\dim c_{I_n}\le \dim c_{I_n\setminus \{n\}}\wedge H_n$. By this and \eqref{eq: use this},
\begin{align}\begin{aligned}
    \widehat{\ell}_2&=N+\max_{\lambda \vdash n-1}\Big\{\big(\dim c_{I_n}+k+1-N\big)+\sum_{I\in \lambda:n\not\in I} \big(\dim c_I+k+1-N\big)\Big\}\\
    &\le \dim c_n+k+1+\max_{\lambda \vdash n}\Big\{\big(\dim c_{I_n\setminus\{n\}}\wedge H_n-\dim c_n\big)+\\
    & \quad\quad\quad\quad\quad\quad\quad\quad\quad\quad\quad\quad\quad\quad\sum_{I\in \lambda: n\not \in I} \big(\dim c_I\wedge H_n-\dim c_n\big)\Big\}.
\end{aligned}
\end{align}
Since $I\in \lambda\setminus \{I_n\}$ together with $I_n\setminus \{n\}$ forms a partition of $[n-1]$, we conclude that $\widehat{\ell}_2$ is bounded from above by $\lck$. We have shown that $\widehat{\ell}\le \lck$, and we proceed by showing the other inequality.

By \Cref{le: subsp dim formula},  
\begin{align}
    \sum_{I\in \lambda} \dim c_I\wedge H_n=\sum_{I\in \lambda} \max\{\dim c_{I\cup\{n\}},\dim c_I+\dim c_n+k+1-N\}.
\end{align}
If the maximum of \eqref{eq: lcprimk} is attained by a partition $\lambda=(I_1,I_2,\ldots,I_s),s\ge 2,$ such that 
\begin{align}
    \dim c_{I_i}\wedge H_n= \dim c_{I_i\cup\{n\}} \textnormal{ for each }i=1,2,
\end{align}
then the parition $\lambda'=(I_1\cup I_2,I_3,\ldots, I_s)$ also attains the maximum of \eqref{eq: lcprimk}. This is a consequence of 
\begin{align}
    -\dim c_n+\dim c_{I_1\cup\{n\}}+\dim c_{I_2\cup\{n\}}\le \dim c_{I_1\cup I_2\cup \{n\}},
\end{align}
which follows from \Cref{le: int + join}. Let $B_{[n-1]}^1$ denote the set of partitions $\lambda$ of $[n-1]$ such that $\dim c_I\wedge H_n=\dim c_{I\cup \{n\}}$ for exactly one $I\in \lambda$, which we denote $I_n$, and let $B_{[n-1]}^0$ denote the set of partitions $\lambda$ for which no $I$ has this property. Then, by the above, in the maximum of \eqref{eq: lcprimk}, we may restrict to partitions $\lambda$ in either $B_{[n-1]}^0$ or $B_{[n-1]}^1$. 

First we optimize over $\lambda\in B_{[n-1]}^0$. \eqref{eq: lcprimk} becomes
    \begin{align}
        \dim c_n+k+1+\max_{\lambda\in B_{[n-1]}^0}\Big\{\sum_{I\in \lambda}\big(\dim c_I+k+1-N\big)\Big\},
    \end{align}
    which equals 
        \begin{align}\label{eq: 123123}
       N+\max_{\lambda\in B_{[n-1]}^0}\Big\{\big(\dim c_n+k+1-N\big)+\sum_{I\in \lambda}\big(\dim c_I+k+1-N\big)\Big\}.
    \end{align}
Letting $\lambda'$ denote the partition consisting of $I\in \lambda$ and $\{n\}$, it is clear that \eqref{eq: 123123} is bounded from above by \eqref{eq: lck part}. Second we optimize over $\lambda\in B_{[n-1]}^1$. \eqref{eq: lcprimk} becomes
    \begin{align}\begin{aligned}
        \dim c_n+k+1+\max_{\lambda\in B_{[n-1]}^1}\Big\{\big(\dim c_{I_n\cup \{n\}}&-\dim c_n\big)+\\
        &\sum_{I\in \lambda\setminus \{I_n\}}\big(\dim c_I+k+1-N\big)\Big\}.
    \end{aligned}
    \end{align}
    This equals 
        \begin{align} \label{eq: 456456}
       N+\max_{\lambda\in B_{[n-1]}^1}\Big\{\big(\dim c_{I_n\cup\{n\}}+k+1-N\big)+\sum_{I\in \lambda}\big(\dim c_I+k+1-N\big)\Big\}.
    \end{align}
    We get a partition of $[n]$ by taking $I\in \lambda\setminus \{I_n\}$ and $I_n\cup\{n\}$. Again, \eqref{eq: 456456} is bounded from above by \eqref{eq: lck part}, proving that $\lck\le\hat{\ell}$ and $\lck=\hat{\ell}$. 
\end{proof}

%%%%%%%%%%%%%%%%%%%%%%%%%%%%%%%%%%%%%%%%%%%%%%%%%%%%%%%%%%%%%%%%%%%%%%%%%%%%%%%%%%%%%%%%%%%%%%%%%%%%%%%%%%%%%%%

Combining \Cref{thm: lck final} with \Cref{prop: dim} directly yields:

\begin{corollary}\label{cor: dim final} We have
\begin{align}
    \dim \Mck=(k+1)\min_{\lambda\vdash n}\Big\{\sum_{I\in \lambda}\big(N-\dim c_I-k-1\big)\Big\}.
\end{align}
\end{corollary}

In \Cref{s: Pseudo-Dis}, we provide a simpler formula for $\lck$ and $\dim \Mck$ under certain conditions on the camera centers. It is in order that we discuss the relation between \Cref{cor: dim final} and the dimension result of Li \cite[Theorem 1.1]{li2018images} in the case of $k=0$ and $c_{[n]}=\emptyset$. For this, we define
\begin{align}\begin{aligned}
    M_{\ca,k}(h):=\{(m_1,\ldots,m_n)\in \NN^n:& \sum_{i=1}^n m_i=h, \textnormal{ and for each } I\subseteq [n],\\
    &\sum_{i\in I}m_i\le (k+1)(N-\dim c_I-k-1)\},
\end{aligned}
\end{align}
where $\NN=\ZZ_{\ge 0}$ denotes the non-negative integers. The following is a generalization of Li's result on dimensions to any $k$.
\begin{theorem}\label{thm: dim other} We have \begin{align}
    \dim \Mck=\max_{h: M_{\ca,k}(h)\neq \emptyset}h.
\end{align}
\end{theorem}

\begin{proof} Let $h$ be the largest integer with $M_{\ca,k}(h)\neq \emptyset$. For $m\in M_{\ca,k}(h)$, and any partition $\lambda$ of $[n]$, we have
\begin{align}
    h=\sum_{I\in \lambda }\sum_{i\in I} m_i\le (k+1)\sum_{I\in \lambda} (N-\dim c_I-k-1),
\end{align}
which proves $\dim \Mck\ge h$, by \Cref{cor: dim final}.

%Let $h$ be the largest integer such that $M(h)\neq \emptyset$. One inequality is clear, namely $\dim \Mck\ge h$. To see this, take $m\in M(h)$ and let $\lambda=([n])$ be the partition consisting of one element. Using \Cref{cor: dim}, $|\lambda|(N-k-1)-\dim_{I\in \lambda} c_I\ge h$, setting $I=[n]$ in the definition of $M(h)$. 

For the other inequality, $\dim \Mck\le h$, we argue via induction. It is easy to check for $n=1$. Then assume it holds for $n-1$, where $n\ge 2$. We fix a generic $H_n$, write $c_i':=c_i\wedge H_n$ and use \Cref{le: lck lcprimk}. We have $\ell_{\ca',k}=\lck$, and
\begin{align}
    h'=(k+1)(\dim c_n+k+1-\lck)
\end{align}
is the largest $h'$ such that $M_{\ca',k}(h')\neq \emptyset$, since $\dim \mathcal M_{\ca',k}=(k+1)(N-\ell_{\ca',k})$. Let $m'=(m_1,\ldots,m_{n-1})\in M_{\ca',k}(h')$, and consider $m_n=(k+1)(N-\dim c_n-k-1)$. It suffices to show that $(m_1,\ldots,m_n)\in M_{\ca,k}(\sum_{i=1}^nm_n)$, since the sum of these $m_i$ equals $\dim \Mck$ by construction, and is bounded from above by $h$ by assumption. Take $I\subseteq [n-1]$, and observe that by induction and \Cref{le: subsp dim formula}, 
\begin{align}\begin{aligned}
    \sum_{i\in I} m_i&\le (k+1)(\dim c_n-\dim c_I\wedge H_n)\\
    &\le (k+1)(\dim c_n-\dim c_I-\dim H_n+N)\\
    &= (k+1)(N-\dim c_I-k-1).
\end{aligned}
\end{align}
Next, take $I\subseteq [n]$ with $n\in I$. Similarly, by the fact that $\dim c_I \le \dim c_{I\setminus \{n\}}\wedge H_n$, we have 
\begin{align}\begin{aligned}
    \sum_{i\in I} m_i=m_n+  \sum_{i\in I\setminus \{n\}} m_i&\le m_n+(k+1)(\dim c_n-\dim c_{I\setminus \{n\}}\wedge H_n)\\
    &\le (k+1)(N-\dim c_I-k-1).
\end{aligned}
\end{align}
This shows that $(m_1,\ldots,m_n)\in M_{\ca,k}(\dim \Mck)$ and we are done.%. This implies that $h\ge \sum_{i=1}^n m_i=(k+1)(N-\lck)$, proving that $\dim \Mck \le h$.
\end{proof}

%%%%%%%%%%%%%%%%%%%%%%%%%%%%%%%%%%%%%%%%%%%%%%%%%%%%%%%%%%%%%%%%%%%%%%%%%%%%%%%%%%%%%%%%%%%%%%%%%%%%%%%%%%%%%%%%%%%%%%%%%%%%%%

Denote by~$L_{d}\subseteq \mathbb P^{h}$ a general subspace of codimension $d$. The \textit{multidegree} of a variety $\mathcal{Y}$ in $\PP^{h_1}\times \cdots \times\PP^{h_m}$ is the function 
\begin{align}
    D(d_{1},\dots,d_{m}):= \#( \mathcal{Y}\cap (L_{d_{1}}^{(1)}\times \cdots\times L_{d_{m}}^{(m)})), 
\end{align}
for $(d_{1},\dots,d_{m})\in \mathbb N^{n}$ such that $d_{1}+\cdots +d_{m} = \operatorname{dim} \mathcal{Y}$. In the context of Computer Vision, multidegrees of a photographic multiview variety $\mathcal M_{\Ca,k}$ tells us how many possible reconstructions of world features are possible given partial information in each image. This is explained in \cite[Section 4]{breiding2023line} in regards to line multiview varieties. In \cite[Theorem 1.1]{li2018images}, it is stated that the multidegrees of the photographic multiview variety $\mathcal M_{\Ca,0}$ are 1 for the tuples $(m_1,\ldots,m_n)$ that lie in $M_{\ca,0}(h)$ for the largest integer $h$ with $M_{\ca,0}(h)\neq \emptyset$, and 0 otherwise. We leave it for future work to provide a similar result on the multidegree for $\mathcal M_{\Ca,k}$ given $k\ge 1$. However, we note that in general for $k\ge 1$, the non-zero multidegrees are not 1 and the sets $M_{\ca,k}(\dim \Mck)$ do not uniquely determine the multidegrees, in contrast to the point case $k=0$:

\begin{example} Let $\Ca$ be a center arrangement of four cameras $\PP^3\dashrightarrow \PP^2$. It follows from the proof of \cite[Theorem 4.1]{breiding2023line}, that as long as the four centers are not collinear, the symmetric $(1,1,1,1)$-multidegree of the line multiview variety $\mathcal L_{\mathcal C}$, as defined in the introduction, is 2. However, we now argue that if the centers are distinct and collinear, then this multidegree is instead 1. Observe that in both cases, the variety is of dimension $4$, and the sets $M(4)$ are equal. Let $E$ be the line containing all centers. Fixing four hyperplanes in each copy of $\Gr(1,\PP^2)$ corresponds to fixing four point $x_i\in \PP^2$ that lines $(\ell_1,\ell_2,\ell_2,\ell_4)\in \mathcal L_{\mathcal C}$ must go through. Recall that any element of $\mathcal L_{\mathcal C}$ has that its back-projected planes meet in a line. We then look for lines in $\PP^3$ that meet all back-projected lines $L_i$ of $x_i$. It is a classic result that there are precisely two lines $K_1$ and $K_2$ in $\PP^3$ meeting four generic lines. In this case, one of these two lines have to be equal to $E$, say $K_1$. The back-projected planes $H_i$ of $\ell_i$ contain the back-projected lines $L_i$ and one of $K_1$ and $K_2$. However, this uniquely defines $H_i$ as the span of $L_i$ and either $K_1$ or $K_2$. Now $H_i=L_i\vee K_1$ gives a tuple that does not lie in $\mathcal L_{\mathcal C}$. Indeed, it follows from \cite{breiding2023line}, that four generic back-projected planes through $E$ do not correspond to an element of the line multiview variety. By genericity, $K_2$ does not meet any centers and $H_i=L_i\vee K_2$ correspond to an element of the line multiview variety. Therefore, in this case the symmetric multidegree is 1 and not 2.\hspace*{\fill}$\diamondsuit\,$     
\end{example}

%%%%%%%%%%%%%%%%%%%%%%%%%%%%%%%%%%%%%%%%%%%%%%%%%%%%%%%%%%%%%%%%%%%%%%%%%%%%%%%%%%%%%%%%%%%%%%%%%%%%%%%%%%%%%%%%%%%%%%%%%%%%%%%%%%%%%%%%%%%%%%%%%%%%%%%%%

%%%%%%%%%%%%%%%%%%%%%%%%%%%%%%%%%%%%%%%%%%%%%%%%%%%%%%%%%%%%%%%%%%%%%%%%%%%%%%%%%%%%%%%%%%%%%%%%%%%%%%%%%%%%%%%%%%%%%%%%%%%%%%%%%%%%%%%%%%%%%%%%%%%%%%%%%%%%%%%%%%%%%%%%%%%%%%%%%%%%%%%%%%%%%%%%%%%%%%%

%%%%%%%%%%%%%%%%%%%%%%%%%%%%%%%%%%%%%%%%%%%%%%%%%%%%%%%%%%%%%%%%%%%%%%%%%%%%%%%%%%%%%%%%%%%%%%%%%%%%%%%%%%%%%%%%%%%%%%%%%%%%%%%%%%%%%%%%%%%%%%%%%%%%%%%%%%%%%%%%%%%%%%%%%%%%%%%%%%%%%%%%%%%%%%%%%%%%%%%

%%%%%%%%%%%

%%%%%%%%%%%

%%%%%%%%%%%

%%%%%%%%%%%%%%%%%%%%%%%%%%%%%%%%%%%%%%%%%%%%%%%%%%%%%%%%%%%%%%%%%%%%%%%%

%%%%%%%%%%%%%%%%%%%%%%%%%%%%%%%%%%%%%%%%%%%%%%%%%%%%%%%%%%%%%%%%%%%%%%%%%%%%%%%%%%%%%%%%%%%%%%%%%%%%%%%%%%%%%%%%

%%%%%%%%%%%%

\section{Pseudo-Disjointedness}\label{s: Pseudo-Dis} One problem that arises when studying generalized multiview varieties, as defined in \Cref{s: GMV}, is that elements of a center arrangements may intersect in complicated ways, obstructing our geometric intuition. To remedy this, we sometimes wish to restrict to subset of all possible center arrangements with more tractable properties. One important insight of this paper is that pseudo-disjointedness is natural choice for such a restriction. For instance, pseudo-disjointed center arrangements imply the expected dimension for $\Mck$, as stated in \Cref{thm: pseudo dim}. Recall the definitions of $\Pck,\Vck$ and $\Hck$ from \Cref{ss: GeoMV}.

\begin{definition} Let $\ca$ be a center arrangement. We say that $\ca$ is:
    \begin{itemize}
       \item \textit{Generic} if each center $c_i$ is generic in $\Gr(\dim c_i,\PP^N)$, 
\item \textit{Disjoint} if the centers $c_i$ are pairwise disjoint,
       \item \textit{Pseudo-disjoint with respect to $k$} if $\Pck\cap \Hck=\Pck\cap \Hck\cap \Vck$.
\end{itemize}
\end{definition}
In some real world scenarios, cameras are be placed along either a fixed axis, or in some other particular pattern. However, generic center arrangements can often be assumed. Generic arrangements are not guaranteed to be disjoint. For instance, generic plane centers in $\PP^4$ meet in a point. As we shall see in the main result of this section, generic center arrangements are pseudo-disjoint with respect to any $k\ge 0$. In addition, disjoint center arrangements are pseudo-disjoint for any $k$, because in this case $\Vck$ is simply the ambient space. 

\begin{example} Consider a linear projection $\PP^2\dashrightarrow \PP^1\times \PP^1$, given by camera matrices with point centers $c_1$ and $c_2$. If the centers coincide, then $(H_1,H_2)\in \mathcal M_{\textbf c,0}$ if and only if $H_1=H_2$, and they contain the centers $c_1=c_2$. Then $\ca=(c_1,c_2)$ is not pseudo-disjoint with respect to $k=0$, because two generic lines through $c_1=c_2$ only meet in a point, meaning $\mathcal P_{\textit{\textbf c},0}\cap \mathcal V_{\textit{\textbf c},0}\cap \mathcal H_{\textit{\textbf c},0}\subsetneq \mathcal P_{\textit{\textbf c},0}\cap \mathcal H_{\textit{\textbf c},0}$. If the centers are distinct, then $\ca$ is a disjoint arrangement, and as explained above is also pseudo-disjoint. \hspace*{\fill}$\diamondsuit\,$   
\end{example}

\begin{example} Consider a linear projection $\PP^3\dashrightarrow \PP^1\times \PP^1$, given by camera matrices with two line centers $c_1$ and $c_2$ meeting in exactly a point. Then, since two planes in $\PP^3$ always meet in a line, for any two back-projected planes $H_1,H_2$, we have $\dim H_1\wedge H_2 \ge 1=\dim c_1\wedge c_2+1$. By definition, $\ca=(c_1,c_2)$ is pseudo-disjoint with respect to $k=0$. \hspace*{\fill}$\diamondsuit\,$ 
\end{example}

The geometric interpretation of pseudo-disjointedness goes as follows. A center arrangement is pseudo-disjoint with respect to $k$ if and only if the intersection of its centers do not imply any conditions on the multiview variety $\Mck$. In other words, pseudo-disjoint arrangements \textit{act} as if they were disjoint. The next two results give alternative characterizations of this property.

\begin{lemma}\label{le: Pck Vck} The center arrangement $\ca$ is pseudo-disjoint with respect to $k$ if and only if 
\begin{align}\label{eq: Pck Vck}
    \Pck=\Pck\cap \Vck.
\end{align}
\end{lemma}

\begin{proof} If $\Pck=\Pck\cap\Vck$, then $\ca$ must be pseudo-disjoint by definition, showing one direction of the statement. For the other, we observe that 
\begin{align}\label{eq: Pck Vck Hck}
     \Pck\cap \Hck=\Pck\cap \Hck\cap \Vck
\end{align}
holds for any subarrangement of $\ca$. Indeed, let $I\subseteq [n]$ and consider the subarrangement $\ca_I:=(c_i)_{i\in I}$. Let $(H_i)_{i\in I}\in \mathcal P_{\ca_I,k}\cap \mathcal H_{\ca_I,k}$ and take any $k$-plane $P\subseteq H_{I}$. Define $H_i$ for $i\in [n]\setminus I$ to be any $\dim c_i+k+1$-planes through $c_i\vee P$. Then $H=(H_i)_{i\in [n]}$ is an element of $\Pck\cap \Hck$, showing by \eqref{eq: Pck Vck Hck} that $H\in \Vck$ and therefore $(H_i)_{i\in I}\in \mathcal V_{\ca_I,k}$. 

Assume by contradiction that $\Pck\supsetneq \Pck\cap \Vck$. Then there is an $H\in \Pck$ and $I\subseteq [n]$ with $c_I\neq \emptyset$ such that $\dim H_I< \dim c_I+k+1$. In particular, this strict inequality occurs for generic $H_i$ through $c_i$, since $\Pck$ is irreducible. We find an $H'\in \mathcal P_{\ca_I,k}$ such that 
\begin{align}
    k\le \dim H_I'< \dim c_I+k+1.
\end{align}
Note that $\dim c_I+k+1\ge k+1$, implying that such a tuple of back-projected planes can indeed exist. It is clearly possible to find $H_i'$ such that $\dim H_I'\ge k$, namely take $H_i'$ to be any $\dim c_i+k+1$-planes through $c_i\vee P$ for any fixed $k$-plane $P$. If this construction yields $\dim H_I'\ge \dim c_I+k+1$, then we may construct a sequence $H_i^{'(a)}\to H_i$ of $\dim c_i+k+1$-planes by starting at $H_i'$ and exchanging one of its spanning vectors to a spanning vector of $H_i$, one at a time. This can be done in a way that preserves the dimension of the back-projected planes. The dimension of $H_I^{'(a)}$ changes at most by one along this sequence. This means that for some $a$, $H_i'':=H_i^{'(a)}$ satisfies 
\begin{align}\label{eq: HI''}
   k\le \dim  H_I''<\dim c_I+k+1.
\end{align}
Then, $(H_i'')_{i\in I}\in \mathcal P_{\ca_I,k}\cap \mathcal H_{\ca_I,k} $, and since \eqref{eq: Pck Vck Hck} holds for $\ca_I$, we have that $(H_i'')_{i\in I}\in \mathcal V_{\ca_I,k}$. This imples that $\dim H_I''\ge \dim c_I+k+1$, which is a contradiction to \eqref{eq: HI''}.
%Since $\dim c_I+k+1\ge k+1$ In particular, this happens for generic $H_i$ through $c_i$, and since $\dim c_I+k+1\ge k+1$, we may choose $H_i$ such that $\dim H_I=k$, and still have that $H\not\in \Vck$.   Choose a generic $k$-plane $P$ containing $c_I$, and choose generic $H_i$ containing $c_i\vee P$.  Since $\dim c_I+k+1\ge k+1$.  In particular, for generic $H_i,i\in I,$ we have $\dim H_I<\dim c_I+k+1$.  One direction is clear. To see the other, note that $\Hck$ only makes sure that back-projected planes meet in $k$ dimensions, but $\Vck$ always requires more. 
\end{proof}

\begin{proposition}\label{prop: Pseudo-Disjoint} The center arrangement $\ca$ is pseudo-disjoint with respect to $k$ if and only if for each $\emptyset \neq I\subseteq [n]$ such that $c_I\neq \emptyset $, we have
    \begin{align}\label{ineq: pseudo}
        \sum_{i\in I} \dim c_i\ge (|I|-1)(N-k-1) + \dim c_{I}.
    \end{align}
If $\textit{\textbf c}$ is generic, then this condition is satisfied for every $k\ge 0$.
\end{proposition}

Recall that for a subspace $V$, a dual $V^*$ is a subspace of dimension $N-\dim V-1$, disjoint from $V$.

\begin{proof} By \Cref{le: Pck Vck}, we may take $\Pck=\Pck\cap \Vck$ as the definition for pseudo-disjointedness.

$ \Leftarrow)$ Let $I\subseteq [n]$ be such that $c_I\neq \emptyset$. \Cref{ineq: pseudo} can be rewritten as
    \begin{align}
        \sum_{i\in I}\big(\dim c_i+k+1\big)\ge (|I|-1)N + \dim c_{I}+k+1.
    \end{align}
By the multiview identities and \Cref{le: intDim}, this means that for any back-projected planes $H_i,i\in I,$ we have that $\dim H_I\ge \dim c_I+k+1$, which shows that $\Pck\subseteq \Pck\cap \Vck$. 

$\Rightarrow)$ Assume by induction that the statement is true for all $N=0,\ldots,N_0$, we show that it holds for $N=N_0+1$. The base case $N=0$ has that $k=0$ and that all centers are empty. In particular, $c_I=\emptyset$ for any $I\subseteq[n]$. 

Choose $I\subseteq [n]$. If $c_I=\emptyset$, there is nothing to prove, and we therefore assume that $c_I\neq \emptyset$. Write $\textit{\textbf c}'$ for the arrangement of centers $c_i':=c_i\wedge c_I^*,i\in I,$ in $c_I^*$. Since $c_I^*$ is a subspace of dimension $N_1:=N-\dim c_I-1$, we may consider $\mathcal P_{\textit{\textbf c}',k}$ as a subset of the ambient space
\begin{align}
     \Gr(\dim c_1'+k+1,\PP^{N_1})\times \cdots\times \Gr(\dim c_n'+k+1,\PP^{N_1}).
\end{align}
We claim that
\begin{align}\label{eq: Pc'k}
    \mathcal P_{\textit{\textbf c}',k}=\mathcal P_{\textit{\textbf c}',k}\cap \mathcal V_{\textit{\textbf c}',k}\cap \mathcal H_{\textit{\textbf c}',k}
\end{align}
if and only if $\textit{\textbf c}_I=(c_i)_{i\in I}$ is pseudo-disjoint with respect to $k$. Before, we prove directions $\Rightarrow$ and $\Leftarrow$ of this claim, note that by \Cref{le: isomo onto *}, we have an isomorphism between elements of $ \mathcal P_{\textit{\textbf c}_{I},k}$ and $\mathcal P_{\textit{\textbf c}',k}$ sending $(H_i)_{i\in I}$ to $(H_i\wedge c_I^*)_{i\in I}$.
\begin{itemize}
\item[$\Rightarrow)$] Take $H\in \mathcal P_{\textit{\textbf c}_{I},k}$, and write $H_i'=H_i\wedge c_I^*$. By \eqref{eq: Pc'k}, we have $\dim H_J'\ge k$ for any $J\subseteq I$, and $\dim H_J'\ge \dim c_J'+k+1$ for any $J\subseteq I$ with $c_J'\neq \emptyset$. In either case, we get $\dim H_J'\ge \dim c_J'+k+1$. By \Cref{le: int join *}, $\dim H_J\ge \dim c_I+\dim c_J'+k+2$. However, by \Cref{le: U vee V*} and the fact that $c_J\subseteq c_I$, we get $\dim c_J'=\dim c_J-\dim c_I-1$. It follows that $\dim H_J\ge \dim c_J+k+1$ for any $J\subseteq I$, showing $H\in \mathcal V_{\textit{\textbf c}_{I},k}$. 

    \item[$\Leftarrow)$] Let $H'\in \mathcal P_{\textit{\textbf c}',k}$ and write $H_i=c_I\vee H_i'$. Then $H\in \mathcal P_{\textit{\textbf c}_{I},k}$, and for any $J\subseteq I$, we have $\dim H_J\ge \dim c_J+k+1$. Here we used that $c_J\neq \emptyset$, since $c_I\neq \emptyset$. By \Cref{le: int join *}, we have $\dim H_J'=\dim H_J-\dim c_I-1$ and as a consequence, $\dim H_J'\ge \dim c_J-\dim c_I+k$. Further, by \Cref{le: uvw}, we have $\dim c_J-\dim c_I=\dim c_J'+1$. This implies that $\dim H_J'\ge \dim c_J'+k+1$ and $H'\in \mathcal V_{\textit{\textbf c}',k}\cap \mathcal H_{\textit{\textbf c}',k}$.
\end{itemize} 

It follows by definition that if $\textit{\textbf c}$ is pseudo-disjoint, then so is $\textit{\textbf c}_I$. By the above, this implies that $\textit{\textbf c}'$ is pseudo-disjoint inside $c_I^*\cong \PP^{N_1}$. Since $N_1\le N_0$, we have by induction 
\begin{align}\label{eq: c_ic_I*}
    \sum_{i\in I} \dim c_i'\ge (|I|-1)(\dim c_I^*-k-1)+\dim c_I',
\end{align}
where $\dim c_I'=-1$. We have $\dim c_I^*=N-\dim c_I-1$, and by \Cref{le: uvw}, we get $\dim c_i=\dim c_i\wedge c_I^*+\dim c_I+1$. Making these substitution, and adding $|I|(\dim c_I+1)$ to both sides of \eqref{eq: c_ic_I*}, we arrive at \Cref{ineq: pseudo} for $I$.

For the last part, we use \Cref{le: subsp dim formula}. Namely, given two generic centers $c_1,c_2$, their intersection may be taken to be a generic subspace of dimension $\dim c_1+\dim c_2-N$. Then the intersection of $c_1\wedge c_2$ with a generic $c_3$ can be taken to be a generic subspace of dimension $\dim c_1+\dim c_2+\dim c_3-2N$, and so on. By reiteration, we get \Cref{ineq: pseudo} for $k=-1$. In particular, it holds for any $k\ge 0$.
\end{proof}

Taking \Cref{ineq: pseudo} as an alternative definition for pseudo-disjointedness, we have that pesudo-disjointedness with respect to $k\le -2$ is equivalent to $\textit{\textbf c}$ being disjoint. Indeed, assume $c_{\{1,2\}}\neq \emptyset$. Then $\dim c_1+\dim c_2\ge N+1+ \dim c_1\wedge c_2$ by \Cref{ineq: pseudo}. It follows by \Cref{le: int + join} that $\dim c_1\vee c_2\ge N+1$ which is a contradiction. Similarly, pseudo-disjointedness for $k=-1$ is equivalent to $c_i$ meeting in the expected dimension, i.e. 
\begin{align}
    \dim c_I=\max\Big\{-1,\Big(\sum_{i=1}^n \dim c_i\Big)-N(n-1)\Big\}.
\end{align}
Let us also consider pseudo-disjointedness for point centers. These are the most common types of centers in application, take for instance the point multiview variety $\mathcal M_{\mathcal C}$ as defined in the introduction. In this case, we can express \Cref{ineq: pseudo} as  
\begin{align}\label{ineq: pseudo points}
     0\ge (|I|-1)(N-k-1),
\end{align}
whenever $c_I\neq \emptyset$, i.e. $\dim c_I=0$. This is equivalent to either $|I|=1$ or $k=N-1$. In the latter case, $c_i\vee P=\PP^N$ for any $k$-plane not meeting $c_i$ and $\Mck=\Pck$ is just a point. Therefore, in this case, pseudo-disjointedness is equivalent to that either $\ca$ is disjoint or $k=N-1$.

\begin{theorem}\label{thm: pseudo dim} Let $\textit{\textbf c}$ be pseudo-disjoint with respect to $k$. Then 
\begin{align}\label{eq: lck}
    \ell_{\textbf c,k}=\max \Big\{k,N-\sum_{i=1}^n\big(h_i-k\big)\Big\}.
\end{align}
\end{theorem}

%%%%%%%%%%%%%%%%%%%%%%%%%%%%%%%%%%%%%%%%%%%%%%%%%%%%%%%%%%%%%%%%%%%%%%%%%%%%%%%%%%%%%%%%%%%%%%%%%%%%%%%%%%%%%%%

\begin{proof} By \Cref{le: subsp dim formula}, we have that $H_i$ meet in at least $\big(\sum_{i=1}^n \dim H_i\big)-(n-1)N$ dimensions, and a direct calculation using the multiview identities shows that this expression equals $N-\sum_{i=1}^n(h_i-k)$. We also know that $\lck\ge k$ always holds. This means that the inequality $\ge$ in \eqref{eq: lck} holds. We are done if we can show that $\le$ also holds. For this we use \Cref{thm: lck final}. The maximum of \eqref{eq: lck part} is attained over partitions $\lambda\vdash n$ such that at most one $I\in \lambda$ has $c_I=\emptyset$. Indeed, if $I_1,I_2\in \lambda$ both have $\dim c_{I_1}=\dim c_{I_2}=\emptyset$, then the partition we get by taking the union of $I_1\cup I_2$ gives a higher value in the right-hand side of \eqref{eq: lck part}. Then, let $B_{[n]}^0$ denote the partitions such that $c_I\neq \emptyset$ for each $I\in \lambda$ and let $B_{[n]}^1$ denote the partitions such that $c_I=\emptyset$ for exactly one $I\in \lambda$. Consider
\begin{align}
    N+\max_{\lambda\in B_{[n]}^0}\Big\{\sum_{I\in \lambda} (\dim c_I+k+1-N)\Big\}.
\end{align}
By \Cref{prop: Pseudo-Disjoint}, it can be estimated from above by 
\begin{align}
    N+\max_{\lambda\in B_{[n]}^0}\Big\{\sum_{I\in \lambda} \sum_{i\in I}(\dim c_i+k+1-N)\Big\}.
\end{align}
This expression may be written as $N-\sum_{i=1}^n(h_i-k)$. Next, for $\lambda\in B_{[n]}^1$, the right-hand side of \eqref{eq: lck part} becomes 
\begin{align}
    N+\max_{\lambda\in B_{[n]}^1}\Big\{k-N+\sum_{I\in \lambda: c_I\neq \emptyset} (\dim c_I+k+1-N)\Big\}.
\end{align}
By \Cref{prop: Pseudo-Disjoint}, it can be estimated from above by 
\begin{align}
    k+\max_{\lambda\in B_{[n]}^1}\Big\{\sum_{I\in \lambda:c_I\neq \emptyset} \sum_{i\in I}(\dim c_i+k+1-N)\Big\}.
\end{align}
Since $\dim c_i+k+1\le N$, this maximum is bounded from above by $k$. Because it was enough to maximize the right-hand side of \eqref{eq: lck part} over $B_{[n]}^0\cup B_{[n]}^1$, this shows that the $\le$ inequality of \eqref{eq: lck}.
\end{proof}

Combining \Cref{prop: dim} and \Cref{thm: pseudo dim} now yields: 

\begin{corollary}\label{cor: dim} Let $\textit{\textbf c}$ be pseudo-disjoint with respect to $k$. Then 
\begin{align}
    \dim \Mck=(k+1)\min\Big\{(N-k),\sum_{i=1}^n\big(h_i-k\big)\Big\}.
\end{align}
\end{corollary}

%%%%%%%%%%%%%%%%%%%%%%%%%%%%%%%%%%%%%%%%%%%%%%%%%%%%%%%%%%%%%%%%%%%%%%%%%%%%%%%%%%%%%%%%%%%%%%%%%%%%%%%%%%%%%%%%%%%%%%%%%%%%%%%%%%%%%%%%%%%%%%%

\section{Properness}\label{s: Prop} This paper has so far studied the question when triangulation is possible for a multiview variety. We take this opportunity to discuss for what center arrangements calibration is possible. In particular, we are interested in when $\Mck=\Pck$. This is equivalent to
\begin{align}
    \mathcal M_{\Ca,k}=\Gr(k,\PP^{h_1})\times \cdots \times \Gr(k,\PP^{h_n}).
\end{align}
If this holds, then any data of an image correspondence $\widetilde{p}$ implies no conditions of $\Ca$, because set-theoretically, there are no constraints on $\mathcal M_{\Ca,k}$. By the following, this property is essentially determined only by geometry of the centers.

\begin{proposition}\label{prop: Spec} We have $\Mck=\Pck$ if and only if $\ca$ is pseudo-disjoint with respect to $k$ and
\begin{align}\label{eq: Spec}
    N-\sum_{i=1}^n(h_i-k)\ge k.
\end{align}
\end{proposition}

\begin{proof} We always have $\Mck\subseteq \Pck$ and since they are both irreducible, $\Mck=\Pck$ is equivalent to $\dim \Mck=\dim \Pck$. If this equality holds, then 
\begin{align}
    \dim \Mck=\sum_{i=1}^n(k+1)\big(h_i-k\big),
\end{align}
since the dimension of the set of $\dim c_i+k+1$-planes through $c_i$ is $(k+1)(N-\dim c_i-k-1)$. It follows by \Cref{prop: dim} that $\lck = N-\sum_{i=1}(h_i-k)$ and since $\lck\ge k$, we conclude that $N-\sum_{i=1}(h_i-k)\ge k$. Further, 
\begin{align}
 \Mck\subseteq \Pck\cap \Hck\cap \Vck\subseteq \Pck=\Mck,  
\end{align}
implying that $\ca$ is pseudo-disjoint with respect to $k$. 

For the other direction, \Cref{thm: pseudo dim} says that 
\begin{align}
    \dim \Mck= \sum_{i=1}^n(k+1)\big(h_i-k\big),
\end{align}
which are stated above is the dimension of $\Pck$, which shows $\Mck=\Pck$.
\end{proof}

By \Cref{prop: Pseudo-Disjoint}, $\Mck=\Pck$ can be rephrased in terms of the intersections of centers and the additional inequality \eqref{eq: Spec}: 

\begin{definition} Let $\ca$ be a center arrangement. We say that $\ca$ is:
    \begin{itemize}
       \item \textit{Proper with respect to $k$} if $\Mck$ is a proper subvariety of $\Pck$, meaning it does not satisfy either \Cref{ineq: pseudo} or \eqref{eq: Spec}.
\end{itemize}
\end{definition}

In application, cameras are manually placed and therefore one usually has a good idea of the corresponding center arrangements. Then for the structure-from-motion pipeline, one should avoid non-proper center arrangements, since they can not be used for calibration, or in other words the specification of camera parameters. We elucidate specifiability in the example below.  

\begin{example} Consider a center arrangement $\ca=(c_1,c_2)$ of two distinct point centers in $\PP^3$. Then $\mathcal M_{\ca,0}\subsetneq \mathcal P_{\ca,0}$ and $\ca$ is proper with respect to $k=0$, since two generic lines in $\PP^3$ do not meet. However, $\mathcal M_{\ca,1}= \mathcal P_{\ca,1}$ and $\ca$ is non-proper with respect to $k=1$, because any two planes in $\PP^3$ meet in a line. \hspace*{\fill}$\diamondsuit\,$ 
\end{example}

It follows from \Cref{prop: Spec,thm: pseudo dim} that in the pseudo-disjoint case, properness implies triangulability. This not true in general, as seen in the next example.

%In partial visibility, specifiability does not imply triangulability, see \cite[Theorem 5]{duff2020pl}. The same is true in complete visibility if $\ca$ is not pseudo-disjoint as seen in the next example. 

\begin{example} Consider a linear projection $\PP^3\dashrightarrow (\PP^1)^n$, given by at least two cameras with distinct line centers $c_i$ all meeting in a mutual point. One can check that $\ell_{\ca,0}=1$ and therefore $\ca$ is not triangulable with respect to $k=0$. Further, since $\dim \mathcal M_{\textbf c,0}=2$, for big enough $n$, the multiview variety cannot equal $\mathcal P_{\ca,0}$, which is of dimension $n$. If $n=3$, then $\ca$ is proper with respect to $k=0$.
\hspace*{\fill}$\diamondsuit\,$ 
\end{example}

In \Cref{s: GMV}, we saw that $\Mck$ was a subset of the intersection $\Pck\cap \Hck\cap \Vck$, but that this inclusion is not always an equality. However, under the conditions of pseudo-disjointedness or triangulability, we can show that it is an irreducible component of $\Pck\cap \Hck\cap \Vck$.

\begin{proposition}\label{prop: irred comp} If $\textit{\textbf c}$ is pseudo-disjoint with respect to $k$ or if $\Mck$ is triangulable, then $\Mck$ is an irreducible component of $\Pck\cap \Hck\cap \Vck$.
\end{proposition}

\begin{proof} If $\ca$ is pseudo-disjoint with respect to $k$ and $\lck>k$, then by \Cref{prop: Spec}, $\Mck=\Pck$, implying that $\Mck$ is equal to $\Pck\cap \Hck\cap \Vck$.

In the case that $\lck=k$, we first prove the inclusion 
\begin{align}\label{eq: right-hand-side}
    \Big(\Pck\cap \Hck\cap \Vck\Big)\setminus \mathcal M_{\ca,k}\subseteq \Big\{(H_1,\ldots,H_n):  c_i\wedge H_{[n]} \neq \emptyset   \textnormal{ for some }i\Big\}.
\end{align}
We do this by taking an element $(H_1,\ldots,H_n)$ in the complement of the right-hand side of \eqref{eq: right-hand-side}, i.e. such that for every $i$, $c_i\wedge H_{[n]}=\emptyset$. It suffices to show that if $(H_1,\ldots,H_n)\in\Pck\cap \Hck\cap \Vck$, then $(H_1,\ldots,H_n)\in \mathrm{Im}\; \Phi_{\ca,k}$. The fact that $c_i\wedge H_{[n]}=\emptyset$, implies that $k$-planes in $H_{[n]}$ do not meet $c_i$. If $P$ is any $k$-plane in $H_{[n]}$, then $(H_1,\ldots,H_n)=\Phi_{\ca,k}(P)$.

Finally, observe that the right-hand side of \eqref{eq: right-hand-side} is a variety and it does not contain $\mathcal M_{\ca,k}$. To see this, take $(H_1,\ldots,H_n)\in  \mathrm{Im}\; \Phi_{\ca,k}$ such that $\dim H_{[n]}=k$. If we write $P=H_{[n]}$, then we must have $(H_1,\ldots,H_n)=\Phi_{\ca,k}(P)$. Now if $c_i\wedge H_{[n]}\neq \emptyset$ for some $i$, then $P$ would meet a center, which is a contradiction.  
\end{proof}

This proof can be extended to show that $\Mck$ is always an irreducible component of $\Pck\cap \mathcal H_{\ca,\ell}\cap \Vck$, where $\ell=\lck$. We leave it is as an open problem to determine whether $\Mck$ is in fact always an irreducible component of $\Pck\cap \Hck\cap \Vck$. A more general problem is to determine precisely when we have $\Mck=\Pck\cap \Vck$. If $\ca$ is pseudo-disjoint, then this property is equivalent to $\Mck=\Pck$, by \Cref{le: Pck Vck}. For instance, if non-triangulability implies $\Mck=\Pck\cap \Vck$, then it would directly follow that $\Mck$ is always an irreducible component of $\Pck\cap \Hck\cap \Vck$.

%%%%%%%%%%%%%%%%%%%%%%%%%%%%%%%%%%%%%%%%%%%%%%%%%%%%%%%%%%%%%%%%%%%%%%%%%%%%%%%%%%%%%%%%%%%%%%%%%%%%%%%%%%%%%%%%%%%%%%%%%%%%%%%%%%%%%%%%%%%%%%%%%%%%%%%%%%%%%%%%%%%%%%%%%%%%%%%%%%%%%%%

\section{Super-Triangulability}\label{s: Super-Tri} There is a natural, stronger condition than triangulability. This is the property that every tuple $H\in \Mck$, and not just generic ones, have a unique $k$-plane $P$ in the intersection of $H_1,\ldots,H_n$. In this direction, for a center arrangement $\ca$ and a non-negative integer $k$, define
\begin{align}
       \upsilon_{\textit{\textbf c},k}:=\max_{(H_1,\ldots,H_n)\in \mathcal M_{\textit{\textbf c},k}} \dim H_{[n]}.
\end{align}

\begin{definition} We say that $\Mck$ is:
\begin{itemize}
   \item \textit{Super-Triangulable} if every tuple $H\in \mathcal M_{\textit{\textbf c},k}$ can be triangulated, meaning $\upsilon_{\ca,k}=k$. 
\end{itemize}
\end{definition}
Our motivation to consider this stronger notion of triangulability is that under this condition, the multiview variety is naturally isomorphic to the associated \textit{multiview blowup}, i.e. the closure $\Gck$ of the graph
\begin{align}
    \Gamma_{\ca,k}:=\{(P,\Phick(P)): P\textnormal{ meets no centers}\}.
\end{align}
Relating multiview varieties to their blowups can be helpful in the study of smoothness and of Chern classes, see for instance \cite[Supplementary Material B3]{trager2015joint}\cite{EDDegree_point}.% via Porteous' formula \cite[Theorem 15.4]{fulton2013intersection}. Examples of this in the Algebraic Vision literature are \cite{trager2015joint,EDDegree_point}.

\begin{proposition}\label{prop: uck isom} We have $\uck=k$ if and only if $\mathcal M_{\ca,k}$ is isomorphic to the blowup $\Gck$ via the projection $\pi: \Gck\to \Mck$ away from the first factor. 
\end{proposition}

\begin{proof} $\pi$ is an isomorphism if and only if it is injective. Indeed, that $\pi$ is injective is equivalent to that for each $(P,H)\in \Gck$, we have $P=H_{[n]}$. In this case, the inverse morphism of $\pi$ sends $H\in \Mck$ to $(H_{[n]},H)\in \Gck$. Further, $\pi$ is injective if and only if $\uck=k$.
\end{proof}

We have the following relations for $\lck$ and $\uck$:
\begin{align}\label{eq: ineqs u l} 
    \dim c_i+k+1\ge \upsilon_{\textit{\textbf c},k}\ge \ell_{\textit{\textbf c},k}\ge \dim c_{[n]}+k
+1 \textnormal{ for each }i.\end{align}
The first inequality follows from the fact that given $H\in \mathcal M_{\textit{\textbf c},k}$, $H_{[n]}$ is contained in each $H_i$, which is of dimension $\dim c_i+k+1$. The second inequality follows by definition. The third inequality is a consequence of \Cref{enum: Hck,enum: Vck} from \Cref{s: GMV}. 

The number $\uck$ has the following natural description.

\begin{proposition}\label{prop: uck first}
    $\upsilon_{\textbf c,k}$ is the largest dimension of a subspace $V$ that meets each center in at least $\dim V-k-1$ dimensions.
\end{proposition}

\begin{proof} First, let $H\in \Mck$ be such that $\dim H_{[n]}=\uck$. Since $H_{[n]}\subseteq H_i$ for every $i$, we have that $H_{[n]}$ meets $c_i$ in at least $\dim H_{[n]}-k-1$ dimensions.

Second, let $V$ be a subspace that meets each center in at least $\dim V-k-1$ dimensions. We show that $\uck\ge \dim V$. By \Cref{le: int + join} and assumption,
\begin{align}
    \dim c_i\vee V=\dim c_i+\dim V-\dim c_i\wedge V\le \dim c_i+k+1.
\end{align}
Let $P$ be a generic $k$-plane in $V$. Then, since $\dim c_i\wedge V\ge \dim V-k-1$, we must have $(c_i\wedge V)\vee P= V$. % and we claim that this is in fact an equality. If $P$ does not meet $c_i$, then
%\begin{align}
%    \dim (c_i\wedge V)\vee P=\dim c_i\wedge V+\dim P+1\ge \dim V,
%\end{align}
%by assumption on $V$. If $P$ does meet $c_i$, then $\dim c_i\wedge V+k+1<\dim V$. In this case, it is however straightforward to see that $(c_i\wedge V)\vee P=V$.
Let $P^{(a)}\to P$ be a sequence of $k$-planes meeting no center and define $H_i^{(a)}:=c_i\vee P^{(a)}$. There is a subsequence such that each $H_i^{(a)}$ converges (in Euclidean topology) and we consider this subsequence. The limits $H_i:=\lim H_i^{(a)}$ contain by construction $(c_i\wedge V)\vee P$ and therefore $V$. Further, $H$ is an element of $\Mck$, showing that $\uck\ge \dim V$.
\end{proof}

Just as we found a formula for $\lck$ in \Cref{s: Tri}, characterizing precisely which multiview varieties are triangulable, we wish to find a formula for $\uck$. We do this using \Cref{prop: uck first} in the case of generic centers. For this, define
\begin{align}\begin{aligned}\label{eq: vck}
  \chi_{\ca,k} &:= \frac{N-1-\sum_{i=1}^n(h_i-k)}{2},\\
    \tau_{\ca,k} &:= \Big\lfloor \chi_{\ca,k} +\sqrt{(k+1)(N-k)+(\chi_{\ca,k}-k)^2}\Big\rfloor,
\end{aligned}
\end{align}
where, $\lfloor \alpha \rfloor$ is the largest integer bounded above by $\alpha\in \RR$. %\FR{Motivate that $\tau$ decreases monotonically}

\begin{theorem}\label{thm: uck} Let $\textit{\textbf c}$ be a generic center arrangement. Then
\begin{align}
      \upsilon_{\ca,k}=\min\Big\{\tau_{\ca,k},\dim c_1+k+1,\ldots,\dim c_n+k+1\Big\}.%\min\Big\{\min_{i\in [m]}\{\dim c_i\}+k+1,\tau\Big\}.%\min \{& \dim c_1+k+1,\ldots, \dim c_m+k+1,\\
      %&\max \Big\{k, \Bigg\lfloor\Bigg(\frac{1}{2(n-1)}+\sqrt{n+\frac{1}{4(n-1)^2}-\sum_{i=1}^m(k+1)(\dim c_i+1)}\Bigg)\Bigg\rfloor\Big\}. 
\end{align}
\end{theorem}

\begin{proof} Consider the set $\mathcal U_{s,i}$ of $s$-planes $V$ that meet $c_i$ in at least $s-k-1$ dimensions for $s$ with $k\le s\le \dim c_i+k+1$. We have
\begin{align}\begin{aligned}
    \dim\;\mathcal U_{s,i}&=(s-k)(\dim c_i+k+1-s)+(k+1)(N-s).   
\end{aligned}
\end{align}
Its codimension in $\Gr(s,\PP^N)$ is $(s-k)(h_i-k)$ by the multiview identities. By \cite[Chapter 4]{eisenbud-harris:16}, the genericity of $c_i$ implies 
\begin{align}\label{eq: cap Usi}
    \dim \bigcap_{i=1}^n\mathcal U_{s,i}=\max \{-1,(s+1)(N-s)-(s-k)\sum_{i=1}^n(h_i-k)\}.
\end{align}
We write $\mathcal U_s$ for the intersection of all $\mathcal U_{s,i}$. We are looking to find the largest $s$ such that $\dim\; \mathcal U_s\ge 0$; this $s$ is $\uck$ by \Cref{prop: uck first}. It is equivalent to find the largest integer $s$ such that 
\begin{align}
    (s+1)(N-s)-(s-k)\sum_{i=1}^n(h_i-k)\ge 0,
\end{align}
which we can rewrite to
\begin{align}\label{eq: second deg}
    s^2-\Big(N-1-\sum_{i=1}^n(h_i-k)\Big)s\le N+\sum_{i=1}^nk(h_i-k).
\end{align}
For big $s$, the left-hand side of \eqref{eq: second deg} is greater than the right-hand side, and for $s=0$, the inequality is satisfied. Setting both sides of \eqref{eq: second deg} equal, we get a degree-2 equation in $s$ that has two real solutions. We deduce that non-negative $s$ that satisfy \eqref{eq: second deg} are characterized by
\begin{align}\begin{aligned}\label{eq: 51}
    0\le s\le &\;\frac{N-1-\sum_{i=1}^n(h_i-k)}{2}\;+\\
    &\sqrt{ N+\sum_{i=1}^nk(h_i-k)+\Big(\frac{N-1-\sum_{i=1}^n(h_i-k)}{2}\Big)^2}.
\end{aligned}
\end{align}
The right-hand side of \eqref{eq: 51} equals the expression inside the floor function of $\tau_{\ca,k}$. Then, the largest integer $s$ that satisfies \eqref{eq: 51} is precisely $\tau_{\ca,k}$. We shown that $\uck=\tau$ as long as $k\le \tau\le \dim c_i+k+1$ for each $i$.

We now prove that $\tau_{\ca,k}\ge k$. We do this by putting $s=k$ in \eqref{eq: cap Usi}. From this we get $\dim \;\mathcal U_s=(s+1)(N-s)\ge 0$, and it follows that \eqref{eq: 51} is satisfied for this choice of $s$. If $\tau \ge \dim c_i+k+1$ for some $i$, let $s$ be the minimum among $\dim c_i+k+1$. We show that $\uck=s$. The integers $s$ satisfies \eqref{eq: 51} and therefore $\uck\ge s$. Conversely, by \eqref{eq: ineqs u l}, $\uck$ is bounded from above by $s$, meaning $\uck=s$.% cannot be bigger than $s$, because $\dim V-k-1\le \dim c_i\wedge V\le \dim c_i$, implies that $ $  It follows that $\dim\; \mathcal U_s\ge 0$, and there exists a subspace $V$ of dimension $s$ meeting all centers in at least $\dim V-k-1$ dimensions, by \Cref{prop: uck first}. However, we must have $\dim V-k-1\le \dim c_i$ for each $i$, implying $s\le \dim c_i+k+1$. Since $s$ was the minimum of all $\dim c_i+k+1$, $\uck$ is the minimum of all $\dim c_i+k+1$.  
\end{proof}

In order to examplify this result, we look at a few familiar cases, namely the standard point multiview and line multiview varieties. In particular, we see that we really need to assume that the center arrangement $\ca$ is generic for \Cref{thm: uck} to hold; it does not suffice to assume disjoint or pseudo-disjoint centers. 

\begin{example} Consider a linear projection $\PP^3\dashrightarrow (\PP^2)^n$, given a generic arrangement $\textit{\textbf c}$ of point centers. For $k=0$, we have \begin{align}
    \tau=\lfloor1-n+ \sqrt{3+(1-n)^2}\rfloor.
\end{align}
For $n=1,2$, $\tau$ equals $1$, while for $n\ge 3$, $\tau$ equals $0$. Since all back-projected planes are lines in this case, $\upsilon_{\ca,0}=\tau$ by \Cref{thm: uck}. This results makes sense with \Cref{prop: uck first}, which tells us that $\upsilon_{\ca,0}=0$ if and only if the centers of $\ca$ are not contained in a line.

Observe that given three point centers that do lie on a line $L$, we get $\uck=1$, because for any point $X\in L$ meeting no centers, $c_i\vee X=L$.\hspace*{\fill}$\diamondsuit\,$   
\end{example}

\begin{example} Consider a linear projection $\PP^3\dashrightarrow (\PP^2)^n$, given a generic arrangement $\textit{\textbf c}$ of point centers. For $k=1$, we have  \begin{align}
    \tau=\Big\lfloor\frac{2-n}{2}+ \sqrt{4+\Big(\frac{2-n}{2} -1\Big)^2}\Big\rfloor.
\end{align}
For $n=1,2,3$, $\tau$ equals $2$, while for $n\ge 4$, $\tau$ equals $1$. Since all back-projected planes are of dimension 2 in this case, $\uck=\tau$. Using \Cref{prop: uck first} directly, we have that $\upsilon_{\ca,1}=1$ if and only if the centers of $\ca$ are not coplanar.\hspace*{\fill}$\diamondsuit\,$  
\end{example}

\begin{example} $\uck$ is not always equal to $\tau$. Indeed, consider the projection $\PP^N\dashrightarrow \PP^{N-1}$, given one point center. Then $\tau=\lfloor \sqrt{N}\rfloor$. However, $\uck=1$, because that is the dimension of a single back-projected line.     \hspace*{\fill}$\diamondsuit\,$ 
\end{example}

%%%%%%%%%%%%%%%%%%%%%%%%%%%%%%%%%%%%%%%%%%%%%%%%%%%%%%%%%%%%%%%%%%%%%%%%%%%%%%%%%%%%%%%%%%%%%%%%%%%%%%%%%%%%%%%%%%%%%%%%%%%%%%%%%%%%%%%%%%%%%%%%%%%%%%%

%%%%%%%%%%%%%%%%%%%%%%%%%%%%%%%%%%%%%%%%%%%%%%%%%%%%%%%%%%%%%%%%%%%%%%%%%%%%%%%%%%%%%%%%%%%%%%%%%%%%%%%%%%%%%%%%%%%%%%%%%%%%%%%%%%%%%%%%%%%%%%%%%%%%%%%%%%%%%%%%

%%%%%%%%%%%%%%%%%%%%%%%%%%%%%%%%%%%%%%%%%%%%%%%%%%%%%%%%%%%%%%%%%%%%%%%%%%%%%%%%%%%%%%%%%%%%%%%%%%

%%%%%%%%%%%%%%%%%%%%%%%%%%%%%%%%%%%%%%%%%%%%%%%%%%%%%%%%%%%%%%%%%%%%%%%%%%%%%%%%%%%%%%%%%%%%%%%%%%

%%%%%%%%%%%%%%%%%%%%%%%%%%%%%%%%%%%%%%%%%%%%%%%%%%%%%%%%%%%%%%%%%%%%%%%%%%%%%%%%%%%%%%%%%%%%%%%%%%%%%%%%%%%%%%%%%%%%%%%%%%%%%%%%%%%%%%%%%%%%%%%%%%%%%%%%%%%%%%%%%%%%%%%%%%%%%

%%%%%%%%%%%%%%%%%%%%%%%%%%%%%%%%%%%%%%%%%%%%%%%%%%%%%%%%%%%%%%%%%%%%%%%%%%%%%%%%%%%%%%%%%%%%%%%%%%%%%%%%%%%%%%%%%%%%%%%%%%%%%%%%%%%%%%%%%%%%%%%%%%%%%%%%
%%%%%%%%%%%%%%%%%%%%%%%%%%%%%%%%%%%%%%%%%%%%%%%%%%%%%%%%%%%%%%%%%%%%%%%%%%%%%%%%%%%%%%%%%%%%%%%%%%%%%%%%%%%%%%%%%%%%%%%%%%%%%%%%%%%%%%%%%%%%%%%%%%%%%%%%%%%%%%%%%%%%%%%%%%%%%%%%%%%%%%%%%%%%%%%%%%%%%%%%%%%%%%%%%%%%%%%%%%%

%%%%%%%%%%%%%%%%%%%%%%%%%%%%%%%%%%%%%%%%%%%%%%%%%%%%%%%%%%%%%%%%%%%%%%%%%%%%%%%%%%%%%%%%%%%%%%%%%%%%%%%%%%%%%%%%%%%%%%%%%%%%%%%%%%%%%%%%%%%%%%%%%%%%%%%%%%%%%%%%%%%%%%%%%%%%%%%%%%%%%%%%%%%%%%%%%%%%%%%%%%%%%%%%%%%%%%%%%%%

\bibliographystyle{alpha}
\bibliography{VisionBib}

%%%%%%%%%%%%%%%%%%%%%%%%%%%%%%%%%%%%%%%%%%%%%%%%%%%%%%%%%%%%%%%%%%%%%%%%%%%%%%%%%%%%%%%%%%%%%%%%%%%%%%%%%%%%%%%%%%%%%%%%%%%%%%%%%%%%%%%%%%%%%%%%%%%%%%%%%%%%%%%%%%%%%%%%%%%%%%%%%%%%%%%%%%%%%%%%

%%%%%%%%%%%%%%%%%%%%%%%%%%%%%%%%%%%%%%%%%%%%%%%%%%%%%%%%%%%%%%%%%%%%%%%%%%%%%%%%%%%%%%%%%%%%%%%%%%%%%%%%%%%%%%%%%%%%%%%%%%%%%%%

\end{document}